\documentclass[aihp]{imsart}

\RequirePackage{amsthm,amsmath,amsfonts,amssymb}
\RequirePackage[numbers]{natbib}
\RequirePackage[colorlinks,citecolor=blue,urlcolor=blue]{hyperref}
\RequirePackage{graphicx}
\usepackage{bbm}

\startlocaldefs
\theoremstyle{plain}

\newtheorem{theorem}{Theorem}[section]
\newtheorem{cor}[theorem]{Corollary}
\newtheorem{lemma}[theorem]{Lemma}
\newtheorem{prop}[theorem]{Proposition}
\theoremstyle{remark}

\newtheorem{rem}[theorem]{Remark}
\newtheorem{example}[theorem]{Example}

\newcommand{\N}{\mathbb{N}}
\newcommand{\R}{\mathbb{R}}
\newcommand{\Z}{\mathbb{Z}}
\newcommand{\E}{\mathbb{E}}
\newcommand{\mmp}{\mathbb{P}}

\DeclareMathOperator{\ind}{\mathbbm{1}}
\newcommand{\Var}{{\rm Var}\,}

\newcommand{\eps}{\varepsilon}

\newcommand{\topr}{\overset{P}{\underset{d\to\infty}\longrightarrow}}
\newcommand{\toasd}{\overset{a.s.}{\underset{d\to\infty}\longrightarrow}}
\newcommand{\toinfd}{{\underset{d\to\infty}\longrightarrow}}
\newcommand{\toprnod}{\overset{P}{\longrightarrow}}
\newcommand{\todistrd}{\overset{w}{\underset{d\to\infty}\longrightarrow}}
\newcommand{\eqdistr}{\stackrel{d}{=}}

\usepackage{MnSymbol}

\newcommand{\cK}{\mathcal{K}}

\newcommand{\cP}{\mathcal{P}}

\newcommand{\bH}{\mathbb{H}}
\newcommand{\aff}{\mathop{\mathrm{aff}}\nolimits}
\newcommand{\conv}{\mathop{\mathrm{conv}}\nolimits}

\newcommand{\codim}{\mathop{\mathrm{codim}}\nolimits}

\endlocaldefs

\endlocaldefs

\begin{document}

\begin{frontmatter}
\title{Random Walks in the High-Dimensional Limit I:\\ The Wiener Spiral}
\runtitle{Random Walks in the High-Dimensional Limit I}

\begin{aug}
\author[A]{\inits{Z.}\fnms{Zakhar}~\snm{Kabluchko}\ead[label=e1]{zakhar.kabluchko@uni-muenster.de}\orcid{0000-0001-8483-3373}}
\and
\author[B]{\inits{A.}\fnms{Alexander}~\snm{Marynych}\ead[label=e2]{marynych@unicyb.kiev.ua}\orcid{0000-0002-7628-7541}}
\address[A]{Institut f\"ur Mathematische Stochastik,
Westf\"alische Wilhelms-Universit\"at M\"unster, M\"unster, Germany\printead[presep={,\ }]{e1}}

\address[B]{Faculty of Computer Science and Cybernetics, Taras Shev\-chen\-ko National University of Kyiv, Kyiv, Ukraine\printead[presep={,\ }]{e2}}
\end{aug}

\begin{abstract}
We prove limit theorems for random walks with $n$ steps in the $d$-dimensional Euclidean space as both $n$ and $d$ tend to infinity.
One of our results states that the path of such a random walk, viewed as a compact subset of the infinite-dimensional Hilbert space $\ell^2$, converges in probability in the Hausdorff distance up to isometry and also in the Gromov-Hausdorff sense to the Wiener spiral, as $d,n\to\infty$. Another group of results describes various possible limit distributions for the squared distance between the random walker at time $n$ and the origin.
\end{abstract}

\begin{abstract}[language=french]
Nous démontrons des théorèmes limite pour des marches aléatoires de longueur $n$ dans l'espace euclidien $d$-dimensionnel, quand $n$ et $d$ tendent tous deux vers l'infini. Nous établissons notamment que la trajectoire de telles marches aléatoires, vue comme un sous-ensemble compact de l'espace de Hilbert $\ell^2$ de dimension infinie, converge en probabilité vers la spirale de Wiener quand $n$ et $d$ tendent vers l'infini, à la fois pour la distance de Hausdorff aux isométries près et la distance de Gromov-Hausdorff. Nous décrivons également les limites en loi possibles pour le carré de la distance entre la marche aléatoire au temps $n$ et l'origine.
\end{abstract}



\end{frontmatter}

\section{Introduction}\label{subsec:introduction}

The purpose of the present paper is to study asymptotic properties of random walks with $n$ steps in the $d$-dimensional space $\R^d$ as both parameters, $n$ and $d$, tend to infinity. To be more concrete, consider a $d$-dimensional random walk  whose increments are independent identically distributed (i.i.d.)\ random vectors with the uniform distribution on the unit sphere $\mathbb S^{d-1}$. In the regime when the dimension $d$ is fixed and the number of steps $n$ tends to infinity, Donsker's invariance principle implies that such random walk converges, after appropriate normalization, to the $d$-dimensional Brownian motion. But how does the path of the random walk look like if $d$ also tends to infinity? It is well known that, as $d\to\infty$, the angle between two independent random vectors sampled uniformly on  the unit sphere $\mathbb S^{d-1}$ tends to $\pi/2$ in probability; see~\cite[Remark~3.2.5]{vershynin_book} or~\cite[Theorem~4]{stam} for stronger results.  This suggests that, informally speaking, the high-dimensional scaling limit of the random walk should be a curve in an infinite-dimensional Hilbert space obtained by gluing together infinitely many mutually orthogonal infinitesimal increments.

A well-known curve of this type is the \emph{Wiener spiral} (or the \emph{crinkled arc}) introduced by Kolmogorov~\cite{kolm}. It is defined as the set  $\{\ind_{[0,t]}: 0\leq t \leq 1\}$ of indicator functions of the intervals $[0,t]$, considered as a subset of the Hilbert space $L^2[0,1]$ and endowed with the induced $L^2$-metric. As a metric space, the Wiener spiral is isometric to the interval $[0,1]$ endowed with the distance $d(t,s)= \sqrt{|t-s|}$. The Wiener spiral can be thought of as a curve $(w_t)_{t\in [0,1]}:= (\ind_{[0,t]})_{t\in [0,1]}$ in the Hilbert space $L^2[0,1]$. It is easy to check that any two ``chords'' $w_y-w_x$ and $w_v-w_u$ with $0\leq x < y \leq  u< v \leq  1$ are orthogonal; see~\cite[Problems~5,6]{Halmos:1982} and~\cite{johnson_crinkled,vitale_crinkled} for results on the uniqueness of the curve having this property.
If $(B_t)_{t\in [0,1]}$ is a standard Brownian motion defined on a probability space $(\Omega, \mathcal F, \mathbb P)$, then the set of random variables $\{B_t: t\in [0,1]\}$, considered as a deterministic subset of $L^2(\Omega,\mathcal F, \mathbb P)$, is isometric to the Wiener spiral.

Let $\ell^2$ be the Hilbert space of square summable real sequences with the standard orthonormal basis $e_1,e_2,\ldots$. The norm on $\ell^2$ will be denoted by $\|\cdot\|$ and the inner product by $\langle \cdot,\cdot\rangle$. For every $d\in \N$ we identify the Euclidean space $\R^d$ with the linear hull of $e_1,\ldots, e_d$ in $\ell^2$, which leads to the sequence of embeddings  $\R^1 \subset \R^2 \subset \cdots \subset \R^d \subset \cdots \subset \ell^2$. This identification will allow us throughout the paper to treat elements of $\R^d$ as elements of $\ell^2$ and use the same notation $\|x\|$ (respectively, $\langle x,y\rangle$) for the usual Euclidean norm of $x\in\R^d$ (respectively, the standard inner product of $x,y\in \R^d$). Note that a continuous curve $\mathbb{W}:=(\widetilde{w}_t)_{t\in[0,1]}$ defined by
\begin{equation}\label{eq:wiener_spiral_l2}
\widetilde{w}_t:=\frac{2\sqrt{2}}{\pi}\sum_{k=1}^{\infty}\frac{\sin(\pi(k-1/2)t)}{2k-1}e_k,\quad t\in [0,1],
\end{equation}
is an isometric realization of the Wiener spiral in the Hilbert space $\ell^2$; see~\cite{vitale_crinkled}.

In which sense one can expect the convergence of random walks (regarded as compact subsets of $\R^d\subset \ell^2$) to the Wiener spiral $\mathbb{W}\subset\ell^2$ to hold? A possible approach, which turns out to be unsatisfactory, is to use the standard Hausdorff distance $d_H$ in $\ell^2$. However, this type of convergence turns out to be too strong for our purposes, since $\ell^2$ contains infinitely many isometric copies of $\mathbb{W}$ and there is no ``natural'' choice for the limiting one. Instead, we shall work with two weaker topologies on the space of compact subsets of $\ell^2$ which in a sense do not distinguish those isometric copies. A first choice is a familiar topology of the {\it Gromov-Hausdorff convergence} of metric spaces. In this setting we regard the path of our random walk as a random metric space endowed with the induced Euclidean distance and prove its convergence to the Wiener spiral, regarded as a deterministic metric subspace of $\ell^2$. The second notion, called the {\it Hausdorff convergence up to isometry}, exploits the fact that the aforementioned metric spaces are already embedded into a common Hilbert space $\ell^2$. We shall show that both notions of convergence are equivalent.

The paper is organized as follows. In Section~\ref{subsec:conv_to_wiener_spiral} we state conditions under which the path of the random walk converges in the Gromov-Hausdorff sense to the Wiener spiral. In Section~\ref{subsec:processes} we prove similar results for high-dimensional continuous-time random processes. In Section~\ref{subsec:hausdorff_up_to_isometry} we discuss Hausdorff convergence up to isometry in $\ell^2$ and state its equivalence to the Gromov-Hausdorff convergence.
In Section~\ref{sec:results_CLT} we shall state results on the limit distribution of the distance between the random walker at time $n$ and the origin. Proofs are collected in Sections~\ref{sec:proof_LLN_wiener_spiral} and~\ref{sec:proofs_CLT}. The present paper deals with random walks having finite second moments. The case of random walks with infinite second moment will be treated in the follow-up work~\cite{KabMar_GH2}. Some bibliographic comments about high-dimensional limits for random walks are collected in Remarks~\ref{rem:biblio1} and~\ref{rem:bibliographic} below.

\section{Convergence to the Wiener spiral}

\subsection{Gromov-Hausdorff convergence of random walks to the Wiener spiral}\label{subsec:conv_to_wiener_spiral}

Recall our convention that $\R^d$ is identified with the linear span of $e_1,\ldots,e_d$ in $\ell^2$. For every $d\in \N$ we consider a random walk in $\R^d$ whose increments $X_1^{(d)}, X_2^{(d)},\dots$ are independent copies of a $d$-dimensional random vector $X^{(d)}$. The random walk is denoted by
\begin{equation}\label{eq:random_walk_def}
S_0^{(d)}:=0,\quad S_k^{(d)}:=X_1^{(d)}+\dots+X_k^{(d)},\quad k\in\N.
\end{equation}
The components of the vectors $X_i^{(d)}$ and $S_{i}^{(d)}$ are denoted by $X_i^{(d)} = (X_{i,1}^{(d)},\dots,X_{i,d}^{(d)})$ and $S_i^{(d)} = (S_{i,1}^{(d)},\dots, S_{i,d}^{(d)})$, respectively.

We impose the following conditions on the increments, which we assume to hold for all $d\in\mathbb{N}$.
\begin{itemize}
\item[(a)] The increments are centered and normalized, that is
\begin{equation}\label{eq:assump_moments}
\E X^{(d)}=0,\quad \E \|X^{(d)}\|^2=1.
\end{equation}
\item[(b)] The components of $X^{(d)}$ are mutually uncorrelated, that is,
\begin{equation}\label{eq:assump_uncor}
\E [ X_{i,j}^{(d)}X_{i,k}^{(d)}] =0,\quad j,k \in \{1,\dots, d\},\quad  j\neq k, \quad i\in\N.
\end{equation}
\item[(c)] The sequence $(\|X^{(d)}\|^2)_{d\in\N}$ is uniformly integrable, that is
\begin{equation}\label{eq:assump_uniform_integrability}
\lim_{A\to\infty} \sup_{d\in \N} \E\left[\|X^{(d)}\|^2 \ind_{\{\|X^{(d)}\|^2>A\}}\right] = 0.
\end{equation}
\item[(d)]  The individual components of $X^{(d)}$ are negligible in the following sense:
\begin{equation}\label{eq:assump_negligible}
\lim_{d\to \infty}\max_{k\in\{1,\dots,d\}}\E (X_{1,k}^{(d)})^2=0.
\end{equation}
\end{itemize}

\begin{example}[Increments with i.i.d.\ components]\label{example11}
Let $\xi_1,\xi_2,\dots$ be i.i.d.\ random variables with $\E \xi_1 = 0$, $\E \xi_1^2 = 1$. If we put $X^{(d)} := (\xi_1,\dots, \xi_d)/\sqrt d$, then conditions (a)--(d) are satisfied.
\end{example}
\begin{example}[Rotationally invariant increments]\label{example12}
Let $X^{(d)}$ be a random vector in $\R^d$ with rotationally invariant distribution. This means that $X^{(d)} = R^{(d)} U^{(d)}$, where $U^{(d)}$ is uniformly distributed on the unit sphere in $\R^d$, and $R^{(d)}\geq 0$ is a random variable independent of $U^{(d)}$. If $\E (R^{(d)})^2 = 1$ for all $d\in \N$ and the sequence $((R^{(d)})^2)_{d\in \N}$ is uniformly integrable, then conditions (a)--(d) are satisfied. In particular, $(R^{(d)})_{d\in\N}$ are allowed to be identically distributed (with finite second moment).
\end{example}

\begin{example}[Random walks jumping along the coordinate axes]\label{example13}
The following model generalizes the simple random walk on $\Z^d$. Let $e_1,\dots, e_d$ denote the standard orthonormal basis of $\R^d$ and let $V^{(d)}$ be a random vector distributed uniformly on the set $\{e_1,\dots, e_d\}$, that is  $\mathbb P[V^{(d)} = e_j]=1/d$ for all $j\in \{1,\dots, d\}$. Put $X^{(d)}:= R^{(d)} V^{(d)}$, where $R^{(d)}$ is a random variable which is independent of $V^{(d)}$  and satisfies $\E R^{(d)} = 0$ and $\E (R^{(d)})^2 = 1$, for all $d\in \N$.  If the sequence $((R^{(d)})^2)_{d\in \N}$ is uniformly integrable, then conditions (a)--(d) are satisfied.  In particular, taking $R^{(d)}$ to be uniformly distributed on $\{+1,-1\}$, we recover the simple symmetric random walk on $\Z^d$.
\end{example}

Let $n=n(d)$ be an arbitrary sequence of positive integers such that $n(d)\to \infty$, as  $d\to\infty$. By default, the notation $d\to\infty$  implies that also $n= n(d)\to\infty$.  We regard the image of the random walk with $n$ steps in $\R^d$ as a finite random metric space. More precisely, let $\mathbb M_d$ be the metric space consisting of the points
\begin{equation}\label{eq:points_random_walk}
0, S_{1}^{(d)}/\sqrt n, \dots, S_{n}^{(d)}/\sqrt n
\end{equation}
and endowed with the metric induced by the Euclidean metric on $\R^d$. Our first main result states that, with probability converging to $1$ as $d\to\infty$, the random metric space $\mathbb M_d$ becomes close, in the sense of the Gromov-Hausdorff distance to be defined below, to the Wiener spiral $\mathbb W$ defined in Section~\ref{subsec:introduction}. Note that $\mathbb W$ is a \textit{deterministic} metric space meaning that, in the high-dimensional limit, the random walk ``freezes'' (i.e., loses its randomness).

The \textit{Gromov-Hausdorff distance} $d_{GH}(E_1,E_2)$ between two compact metric spaces $E_1$ and $E_2$ is defined as the infimum of $d_H(\varphi_1(E_1),\varphi_2(E_2))$ taken over all metric spaces $(M,\rho)$ as well as  all isometric embeddings $\varphi_1: E_1\to M$ and $\varphi_2:E_2\to M$, and $d_H$ denotes the Hausdorff distance between compact subsets of $M$ defined by
$$
d_H(A,B) = \inf\{r>0: A\subset U_r(B), B\subset U_r(A)\}.
$$
Here, $U_r(A) = \{m\in M: \rho(A, m) < r\}$ is the $r$-neighborhood of $A$ in $M$. For details, we refer to Chapter~7 of~\cite{Burago+Burago:2001}. It is known that the set of isometry classes of compact metric spaces, endowed with the Gromov-Hausdorff distance, becomes a complete separable metric space, called the \emph{Gromov-Hausdorff space}. We are now ready to state our first  result.

\begin{theorem}\label{theo:GH_conv_to_wiener_spiral}
Let $n=n(d)$ be an arbitrary sequence of positive integers such that $n(d)\to \infty$, as  $d\to\infty$. Suppose that conditions (a)--(d) are fulfilled. Then, as $d\to\infty$, the random metric space $\mathbb M_d$,  considered as a random point in the Gromov-Hausdorff space, converges in probability to the Wiener spiral $\mathbb W$. That is to say, for every $\eps>0$,
$$
\lim_{d\to\infty} \mathbb P[d_{\text{GH}}(\mathbb M_d, \mathbb W) >\eps] = 0.
$$
\end{theorem}
The proof of Theorem~\ref{theo:GH_conv_to_wiener_spiral} will be given in Sections~\ref{subsec:proof_LLN_norm_S_n} and~\ref{subsec:proof_conv_to_wiener_spiral}. We shall also verify that the claim stays in force if $\mathbb M_d$ is replaced by the polygonal line interpolating consecutive points in~\eqref{eq:points_random_walk}, that is, for the metric space $\mathbb{M}^{\mathrm{cont}}_d$ given by
$$
\mathbb{M}^{\mathrm{cont}}_d:=\bigcup_{i=0}^{n-1}\left[ \frac{S_i^{(d)}}{\sqrt{n}},\frac{S_{i+1}^{(d)}}{\sqrt{n}}\right],
$$
where $[a,b]\subset\R^d$ is the closed segment connecting $a,b\in\R^d$. As before, the space $\mathbb{M}^{\mathrm{cont}}_d$ is endowed with the induced Euclidean metric on $\R^d$.
\begin{cor}\label{cor:linear_interpolation}
Under the same assumptions as in~Theorem~\ref{theo:GH_conv_to_wiener_spiral}, for every $\eps>0$,
$$
\lim_{d\to\infty} \mathbb P[d_{\text{GH}}(\mathbb{M}^{\mathrm{cont}}_d, \mathbb W) >\eps] = 0.
$$
\end{cor}

\subsection{Gromov-Hausdorff convergence of high-dimensional stochastic processes}\label{subsec:processes}
In this section we state a result which is similar in spirit to Theorem~\ref{theo:GH_conv_to_wiener_spiral} but applies to a different class of stochastic processes. Let $K$ be an arbitrary index set and $X= (X(t))_{t\in K}$ be a real-valued stochastic process with $\E X(t) = 0$ and $\E (X(t))^2 <\infty$ for all $t\in K$. We suppose that $\rho(s,t) := \sqrt{\Var (X(s) - X(t))}$ defines a metric on $K$ which turns $K$ into a compact metric space and that the process $X$ has a.s.\ continuous sample paths on $(K,\rho)$. Finally, we suppose that  $\E [\sup_{t\in K} (X(t))^2] <\infty$. Let $(X_1(t))_{t\in K},(X_2(t))_{t\in K}, \dots$ be independent copies of the process $X$.
For every $d\in \N$ we consider the $\R^d$-valued stochastic process
$$
\mathbb X_{d}(t) := \frac {(X_1(t), \dots, X_d(t))} {\sqrt d} \in \R^d, \qquad t\in K.
$$
\begin{theorem}\label{theo:iid_components_wiener_spiral}
The random metric space $\mathbb K_d:= \{\mathbb X_{d}(t): t\in K\}\subset \R^d$,
endowed with the induced Euclidean metric, converges a.s.\ (as $d\to\infty$) to the deterministic metric space $(K, \rho)$ in the Gromov-Hausdorff sense. That is to say,
$$
\mathbb P\left[\lim_{d\to\infty} d_{\text{GH}}(\mathbb K_d, K)=0 \right] = 1.
$$
\end{theorem}
\begin{example}
Let $(X(t))_{t\in [0,1]}$ be the standard Brownian motion. Then, $(\mathbb X_{d}(t))_{t\in [0,1]}$ is a standard $d$-dimensional Brownian motion multiplied by $1/\sqrt d$.  Theorem~\ref{theo:iid_components_wiener_spiral} implies that the random metric space $\{\mathbb X_{d}: t\in [0,1]\} \subset \R^d$, viewed as a random point in the Gromov-Hausdorff space,  converges a.s.\ to the Wiener spiral $\mathbb W$.
\end{example}

\subsection{Hausdorff distance up to isometry in \texorpdfstring{$\ell^2$}{l2}}\label{subsec:hausdorff_up_to_isometry}
Let $\cK(\ell^2)$ be the space of compact subsets of $\ell^2$ endowed with the Hausdorff metric $d_{H}$. We introduce the following equivalence relation $\sim$ on $\cK(\ell^2)$. Two compact subsets $K_1\subset \ell^2$ and $K_2\subset \ell^2$ are considered equivalent if there is an isometry $J:K_1\to K_2$, that is a bijection between $K_1$ and $K_2$ that preserves distances. Note that $J$ need not be defined outside $K_1$. More generally, for the purpose of this paper, an \textit{isometry} between two metric spaces is a \textit{bijection} between these spaces that preserves distances. An \textit{injective} map preserving distances is called \textit{isometric embedding}.

The next lemma is standard; see~\cite[Theorem 11.4]{Wells:2012}. We shall provide a self-contained proof in Appendix~\ref{appD}.
\begin{lemma}\label{lem:isometry_extension_to_affine_hull}
Any isometry $J: K_1 \to K_2$ can be extended to a unique isometry $J:\overline{\aff} K_1 \to \overline{\aff} K_2$.
\end{lemma}
Here, $\aff K$ is the minimal affine subspace containing a set $K\subset \ell^2$, and $\overline{\aff} K$ is the closure of $\aff K$. Note that $\overline{\aff} K$ is a closed affine subspace of $\ell^2$ (that is, a parallel translate of a closed linear subspace). The extended isometry $J:\overline{\aff} K_1 \to \overline{\aff} K_2$ is an affine map. In fact, it is well-known that any isometry between Hilbert spaces is an affine map. In general, an isometry $J:K_1\to K_2$ between two compact subsets $K_1\subset \ell^2$ and $K_2\subset \ell^2$ need not admit an  extension to a \textit{global} self-isometry of the whole Hilbert space $\ell^2$, as the following example demonstrates.

\begin{example}\label{example:two_isometric_sets_no_global_isometry}
Put $K_1:= \{0\} \cup\{e_n/n: n\in \N\}$ and $K_2:= \{0\} \cup \{e_{n+1}/n: n\in \N\}$. The map $J(0) := 0$, $J(e_n/n):= e_{n+1}/n$ defines an isometry between $K_1$ and $K_2$. However, $\overline{\aff} K_1 = \ell^2$ while $\overline{\aff} K_2$ is the orthogonal complement of $e_1$ (and hence a proper subset of $\ell^2$). From the uniqueness part of Lemma~\ref{lem:isometry_extension_to_affine_hull} it follows that we cannot extend $J$ to a global self-isometry of $\ell^2$. Below, see the proof of Proposition~\ref{prop:tilde_d_is_a_metric}, we shall show that an isometry $J:K_1\to K_2$ can always be extended to a global self-isometry of $\ell^2$ if $\codim (\overline{\aff} K_1)=\codim (\overline{\aff} K_2)=\infty$.
\end{example}

The equivalence class of a compact set $K$ is denoted by $[K]:=\{K'\in \cK(\ell^2): K\sim K'\}$. The set of all such equivalence classes is denoted by $\bH := \cK(\ell^2)/\sim$. Now we introduce a metric on $\bH$. For $K_1,K_2\in \cK(\ell^2)$, the \textit{Hausdorff distance up to isometry} between $[K_1]$ and $[K_2]$ is defined by
$$
d_{\sim}([K_1], [K_2]) = \inf_{K_1'\in [K_1], K_2'\in [K_2]} d_{H}(K_1', K_2').
$$
\begin{prop}\label{prop:tilde_d_is_a_metric}
The function $d_{\sim}:\bH\times \bH\mapsto [0,\infty)$ is a metric on $\bH$.
\end{prop}

The proof of Proposition~\ref{prop:tilde_d_is_a_metric} will be given in Section~\ref{sec:proof_up_to_isometry}.

\begin{rem}\label{rem:GH_metric_vs_Hilbert_GH}
It follows directly from the definition that for every pair of compact sets $K_1, K_2\subset \ell^2$ regarded as metric spaces with the induced $\ell^2$-metric, we have
$$
d_{\sim}([K_1],[K_2]) \geq d_{GH}(K_1, K_2).
$$
Therefore, convergence in the Hausdorff distance up to isometry implies convergence in the Gromov-Hausdorff sense. In fact, the other direction is also true, as the next theorem shows.
\end{rem}

\begin{theorem}\label{theo:equivalence_GH_and_up_to_iso}
Let $K_1,K_2,\ldots$ and $K$ be compact subsets of $\ell^2$. Then, $[K_n] \to [K]$ in $(\bH, d_{\sim})$ if and only if $K_n\to K$ in the Gromov-Hausdorff sense (where $K_n$ and $K$ are regarded as metric spaces with the induced $\ell^2$-metric).
\end{theorem}
The proof of Theorem~\ref{theo:equivalence_GH_and_up_to_iso} will be given in Section~\ref{sec:proof_up_to_isometry}.
\begin{rem}
A natural question arising from the definition of the Hausdorff distance up to isometry is why not to declare two sets $K_1,K_2\subset \ell^2$ to be equivalent if there exists a {\it global} isometry $J:\ell^2\to \ell^2$ (a self-bijection of $\ell^2$ preserving distances) such that $J(K_1) = K_2$? Let us discuss such a possibility. Any global self-isometry of $\ell^2$ has the form $Jx = Ox + b$ for some bijective transformation $O:\ell^2 \to \ell^2$ which preserves the inner product (orthogonal transformation) and some vector $b\in \ell^2$. The global self-isometries of $\ell^2$ form a group.   It is easy to check that $\approx$ is an equivalence relation which, by Example~\ref{example:two_isometric_sets_no_global_isometry}, is different from $\sim$. Let $\lsem K\rsem$ be the equivalence class of $K$ with respect to $\approx$. One easily checks that
$$
d_{\approx} (\lsem K_1 \rsem, \lsem K_2\rsem) := \inf_{K_1' \in \lsem K_1 \rsem, K_2' \in \lsem K_2 \rsem} d_{H} (K_1', K_2')
$$
defines a \textit{pseudometric} on the set of equivalence classes with respect to $\approx$. However, $d_{\approx}$ is not a metric. Indeed, the sets $K_1$ and $K_2$ defined in Example~\ref{example:two_isometric_sets_no_global_isometry} are not equivalent in the sense of $\approx$. On the other hand, for every $n\in \N$ we can consider the following isometry  $I_n:\ell^2\to \ell^2$:
$$
I_n(e_{k+1}) = e_k,
\qquad
1\leq k \leq n,
\qquad
I_n(e_1) = e_{n+1},
\qquad
I_n(e_\ell) = e_{\ell},
\qquad
\ell \geq n+2.
$$
Then,
$$
I_n(K_2) =  \left\{0,\frac {e_1}{1}, \frac{e_2}{2},\ldots, \frac {e_n}{n}, \frac{e_{n+2}}{n+1}, \frac{e_{n+3}}{n+2},\ldots\right\},
\qquad
K_1 = \left\{0,\frac {e_1}{1}, \frac{e_2}{2},\ldots, \frac {e_n}{n}, \frac{e_{n+1}}{n+1}, \frac{e_{n+2}}{n+2},\ldots\right\}.
$$
It follows that $d_{H} (I_n(K_2), K_1) \to 0$ as $n\to\infty$, which proves that $d_\approx(\lsem K_1 \rsem , \lsem K_2 \rsem) = 0$ even though $\lsem K_1 \rsem \neq \lsem K_2 \rsem$ by Example~\ref{example:two_isometric_sets_no_global_isometry}. Therefore, it seems more natural to consider $\sim$ (and the corresponding metric $d_{\sim}$) rather than $\approx$ (and the corresponding pseudometric $d_{\approx}$).
\end{rem}

The next result is an immediate consequence of Theorems~\ref{theo:GH_conv_to_wiener_spiral} and~\ref{theo:equivalence_GH_and_up_to_iso}.
\begin{theorem}\label{theo:conv_to_wiener_spiral_in_hilbert_up_to_isometry}
Let the assumptions of Theorem~\ref{theo:GH_conv_to_wiener_spiral} be satisfied.
Then, $[\mathbb M_d]$ converges in probability to $[\mathbb W]$ in $(\bH, d_{\sim})$, as $d\to\infty$. That is, for every $\eps>0$,
\begin{equation}\label{eq:theo:conv_to_wiener_spiral_in_hilbert_up_to_isometry}
\lim_{d\to\infty} \mathbb P\left[ d_{\sim} ([\mathbb M_d], [\mathbb W])>\eps\right] = 0.
\end{equation}
\end{theorem}

Using Theorem~\ref{theo:conv_to_wiener_spiral_in_hilbert_up_to_isometry} it will be easy to deduce the following result on convergence of convex hulls of random walks. Let $\conv K$ denote the convex hull of a subset $K\subset \ell^2$ and $\overline{\conv} K$ denote its closure.

\begin{theorem}\label{theo:conv_to_convex_hull_of_wiener_spiral_in_hilbert_up_to_isometry}
In the setting of Theorem~\ref{theo:conv_to_wiener_spiral_in_hilbert_up_to_isometry}, $[\conv \mathbb{M}_d]$ converges to $[\overline{\conv}\, \mathbb W]$ in $(\bH, d_{\sim})$ in probability, as $d\to\infty$. That is, for every $\eps>0$, we have
$$
\lim_{d\to\infty} \mathbb P\left[d_{\sim}([\conv \mathbb{M}_d],  [\overline{\conv} \, \mathbb{W}])>\eps\right] = 0.
$$
\end{theorem}

\begin{rem}
Note that $\overline{\conv} \mathbb{W}$ is isometric to the set of nondecreasing functions $f\in L^2[0,1]$ with $0\leq f(x) \leq 1$ for all $x\in [0,1]$.   Also note that $\conv \mathbb{M}_d$ is already closed since the set $\mathbb{M}_d$ is finite.
\end{rem}

By Remark~\ref{rem:GH_metric_vs_Hilbert_GH}, Theorem~\ref{theo:conv_to_convex_hull_of_wiener_spiral_in_hilbert_up_to_isometry} also implies convergence in the Gromov-Hausdorff sense.

\section{Central limit theorems for the squared norm}\label{sec:results_CLT}
In this section we state distributional limit theorems for the squared norm $\|S_n^{(d)}\|^2$, as $d\to\infty$, in the three models presented in
Examples~\ref{example11}, \ref{example12} and~\ref{example13} of Section~\ref{subsec:conv_to_wiener_spiral}.
Recall that all the corresponding random walks satisfy the assumptions (a)-(d) and, thus, converge to the Wiener spiral.
However, the distributional behaviour is more sensitive to the details of each of the models and the corresponding distributional
limit theorems are different. As before, $n=n(d)$ is an arbitrary sequence of positive integers such that $n(d)\to \infty$, as  $d\to\infty$.

\subsection{Model 1: Random walks whose increments have i.i.d.\ components}\label{subsec:model_1}
Recall that in this model $(\xi_{i,j})_{i,j=1}^\infty$ are independent copies of a random variable $\xi$ such that $\E \xi = 0$ and $\E \xi^2 = 1$, and for every $d\in \N$ the increments of a $d$-dimensional random walk~\eqref{eq:random_walk_def} are given by
$$
X_i^{(d)} := \frac {(\xi_{i,1},\dots, \xi_{i,d})} {\sqrt{d}} ,
\quad i\in\N.
$$
\begin{theorem}\label{theo:distr_norm_model_1}
In the setting just described suppose additionally that $\E \xi^4 < \infty$.  Then,
$$
\frac{\|S_n^{(d)}\|^2 - n}{ \sqrt{2n^2/d}} \todistrd \mathrm{N}(0,1).
$$
\end{theorem}

Here and in what follows, $\mathrm{N}(0,\sigma^2)$ denotes the centered normal distribution with variance $\sigma^2$, and $\overset{w}{\longrightarrow}$ denotes weak convergence of probability measures (convergence in distribution). In the next theorem we treat the case when $\xi^2$ has infinite second moment. More precisely, we suppose that $\xi^2$ belongs to the domain of attraction of an $\alpha$-stable distribution with $\alpha \in (1,2)$. This means that the independent copies of $\xi$, denoted by $(\xi_{i})_{i=1}^\infty$, satisfy
\begin{equation}\label{eq:alpha_stable_domain_attraction}
\frac{\xi_1^2 + \dots + \xi_{m}^2 - m}{m^{1/\alpha} L(m)}
\overset{w}{\underset{m\to\infty}\longrightarrow}
\zeta_\alpha
\end{equation}
for some slowly varying function $L$ and a zero-mean random variable $\zeta_\alpha$ having a spectrally positive $\alpha$-stable distribution.
\begin{theorem}\label{theo:distr_norm_model_1_stable}
Suppose that~\eqref{eq:alpha_stable_domain_attraction} holds for some $\alpha\in (1,2)$.
\begin{enumerate}
\item[(a)] If $n > d^{\frac{2 - \alpha}{2\alpha  - 2} + \delta}$ for some $\delta>0$ and all sufficiently large $d$, then
$$
\frac{\|S_n^{(d)}\|^2 - n}{ \sqrt{2n^2/d}} \todistrd \mathrm{N}(0,1).
$$
\item[(b)] If $n < d^{\frac{2 - \alpha}{2\alpha  - 2} - \delta}$ for some $\delta >0$ and all sufficiently large $d$, then
$$
\frac{\|S_n^{(d)}\|^2 - n}{d^{-1}(nd)^{1/\alpha} L(nd)} \todistrd \zeta_\alpha.
$$
\end{enumerate}
\end{theorem}

\subsection{Model 2: Random walks with rotationally invariant increments}\label{subsec:model_2}
We shall further specialize Example~\ref{example12} by assuming additionally that the distribution of $R^{(d)}$ is the same for all $d\in\mathbb{N}$. Thus, for every $d\in \N$ we consider a random walk~\eqref{eq:random_walk_def} in $\R^d$ whose increments are given by
$$
X_i^{(d)} := R_i U_i^{(d)}, \quad i\in\N,
$$
where
\begin{itemize}
\item the radial components $(R_i)_{i=1}^\infty$ are independent copies of a non-negative random variable $R$ with $\E R^2 = 1$;
\item the directional components $(U_i^{(d)})_{i=1}^\infty$ are i.i.d.\ random vectors uniformly distributed on the unit sphere in $\R^d$;
\item $(R_i)_{i=1}^\infty$ and $(U_i^{(d)})_{i=1}^\infty$ are independent.
\end{itemize}
\begin{theorem}\label{theo:distr_norm_model_2}
In the setting just described, suppose additionally that $\E R^4 <\infty$.
\begin{enumerate}
\item[(a)]
If $\lim_{d\to\infty} n/d = 0$ and $R$ is not deterministic, then
$$
\frac{\|S_n^{(d)}\|^2 - n}{\sqrt{n}} \todistrd \mathrm{N}(0,\Var (R^2)).
$$
\item[(b)] If $\lim_{d\to\infty} n/d = \infty$ or $R$ is deterministic, then
$$
\frac{\|S_n^{(d)}\|^2 - n}{ \sqrt{2n^2/d}} \todistrd \mathrm{N}(0,1).
$$
\item[(c)] If $n\sim \gamma  d$ for some constant $\gamma \in (0,\infty)$, then
$$
\frac{\|S_n^{(d)}\|^2 - n}{ \sqrt{n}} \todistrd \mathrm{N}(0,2\gamma + \Var (R^2)).
$$
\end{enumerate}
\end{theorem}
\begin{rem}
Let us mention known results for random walks with \textit{fixed} number of steps.
Stam~\cite[Theorem~4 and p.~227]{stam} showed that if $R=1$ is deterministic and $m\in \N$ is fixed, then
$$
\frac{\|S_m^{(d)}\|^2 - m}{ \sqrt{2m(m-1)/d}} \todistrd \mathrm{N}(0,1).
$$
On the other hand, if $R$ is not deterministic and $m\in \N$ is fixed, then it follows from~\cite[Theorem~4]{stam} that
$$
\|S_m^{(d)}\|^2 \todistrd R_1^2 + \dots + R_m^2.
$$
\end{rem}

\begin{rem}\label{rem:biblio1}
After this paper was finished the authors learned that Theorem~\ref{theo:distr_norm_model_2} has been proved under the same conditions in~\cite{grundmann2014limit}; see also~\cite{grundmann_thesis}. Furthermore, some particular cases have been known before; see for example~\cite{Voit+2012+231+246}. Our method of proof is based on martingale techniques and is completely different from the methods used in~\cite{grundmann2014limit}. The accompanying laws of large numbers have been derived in~\cite{rosler2011limit}.
\end{rem}

Let us now consider the case when  $R^2$ belongs to the domain of attraction of an $\alpha$-stable distribution with $\alpha \in (1,2)$ meaning that
\begin{equation}\label{eq:alpha_stable_domain_attraction_model_2}
\frac{R_1^2 + \dots + R_{n}^2 - n}{n^{1/\alpha} L(n)}
\overset{w}{\underset{n\to\infty}\longrightarrow}
\zeta_\alpha
\end{equation}
for some slowly varying function $L$ and a zero-mean random variable $\zeta_\alpha$ having a spectrally positive $\alpha$-stable distribution.
\begin{theorem}\label{theo:distr_norm_model_2_stable}
Suppose that~\eqref{eq:alpha_stable_domain_attraction_model_2} holds for some $\alpha\in (1,2)$.
\begin{enumerate}
\item[(a)] If $n > d^{\frac{\alpha}{2\alpha  - 2} + \delta}$ for some $\delta>0$ and all sufficiently large $d$, then
$$
\frac{\|S_n^{(d)}\|^2 - n}{ \sqrt{2n^2/d}} \todistrd \mathrm{N}(0,1).
$$
\item[(b)] If $n < d^{\frac{\alpha}{2\alpha  - 2} - \delta}$ for some $\delta >0$ and all sufficiently large $d$, then
$$
\frac{\|S_n^{(d)}\|^2 - n}{n^{1/\alpha} L(n)} \todistrd \zeta_\alpha.
$$
\end{enumerate}
\end{theorem}
We shall comment on the missing ``critical'' case of this theorem in Remark~\ref{rem:critical_case_stable_model_2}.

\subsection{Model 3: Random walks jumping along the coordinate axes}\label{subsec:model_3}
As we did in the previous model, here we also impose an additional assumption in the setting of Example~\ref{example13} and suppose that the distribution of $R^{(d)}$ is the same for all $d\in\mathbb{N}$. Thus, for every $d\in \N$ we consider a random walk~\eqref{eq:random_walk_def} in $\R^d$ whose increments are given by
$$
X_i^{(d)} := R_i V_i^{(d)},
\quad
i\in\N,
$$
where
\begin{itemize}
\item  $(R_i)_{i=1}^\infty$ are independent copies of a random variable $R$ with $\E R = 0$ and  $\E R^2 = 1$.
\item $(V_i^{(d)})_{i=1}^\infty$ are i.i.d.\ random vectors uniformly distributed on $\{e_1,\dots,e_d\}$, the standard orthonormal basis of $\R^d$. That is to say,
$$
\mathbb P[V_i^{(d)} = e_j] =1/d, \quad i\in \N,\quad  j\in \{1,\dots, d\}.
$$
\item $(R_i)_{i=1}^\infty$ and $(V_i^{(d)})_{i=1}^\infty$ are independent.
\end{itemize}
This model is related to an experiment in which $n$ balls are independently placed into $d$ equiprobable boxes. If the $i$-th ball is placed into box $j$, then the $i$-th increment of the random walk is equal to $R_i e_{j}$.  
\begin{theorem}\label{theo:distr_norm_model_3}
In the setting just described suppose that $\E R^4 <\infty$.
\begin{enumerate}
\item[(a)] If $\lim_{d\to\infty} n/d = 0$, then
$$
\frac{\|S_n^{(d)}\|^2 - n}{\sqrt{n}} \todistrd \mathrm{N}(0,\Var (R^2)).
$$
\item[(b)] If $\lim_{d\to\infty} n/d= \infty$, then
$$
\frac{\|S_n^{(d)}\|^2 - n}{ \sqrt{2n^2/d}} \todistrd \mathrm{N}(0,1).
$$
\item[(c)] If $n\sim \gamma  d$ for some constant $\gamma \in (0,\infty)$, then
$$
\frac{\|S_n^{(d)}\|^2 - n}{ \sqrt{n}} \todistrd \mathrm{N}(0,2\gamma + \Var(R^2)).
$$
\end{enumerate}
\end{theorem}

In the case when $R^2$ belongs to the domain of attraction of an $\alpha$-stable distribution with $\alpha \in (1,2)$, the conclusion is identical to that of Theorem~\ref{theo:distr_norm_model_2_stable}.
\begin{theorem}\label{theo:distr_norm_model_3_stable}
If~\eqref{eq:alpha_stable_domain_attraction_model_2} holds in the setting of Model~3, then the same conclusions as in Theorem~\ref{theo:distr_norm_model_2_stable} apply.
\end{theorem}

Note that the conclusions of Theorems~\ref{theo:distr_norm_model_2} and~\ref{theo:distr_norm_model_3} are almost identical, the only difference being that the latter does not provide a precise answer in the case of deterministic $R$ in the regime $\lim_{d\to\infty}n/d=0$, since the limit in Part (a) is then degenerate.  The next theorem gives a more precise result in this case. Without loss of generality, we assume that $R^2=1$. The latter in conjunction with $\mathbb{E}R=0$ implies that $(S_i^{(d)})_{i=0}^\infty$ must be the simple symmetric random walk.
\begin{theorem}\label{theo:distr_norm_model_3_simple_RW}
Let $(S_i^{(d)})_{i=0}^\infty$ be the simple symmetric random walk on $\Z^d$ starting at $0$.
\begin{enumerate}
\item[(a)] If $n= o(\sqrt d)$, then $\lim_{d\to\infty} \mathbb P[\|S_n^{(d)}\|^2 = n] = 1$.
\item[(b)] If $n\sim c\sqrt d$ for some constant $c\in (0,\infty)$, then
$$
\|S_n^{(d)}\|^2 - n \todistrd 3 P' - P'',
$$
where $P'$ and $P''$ are independent Poisson random variables with mean $c^2/4$.
\item[(c)] If $\lim_{d\to\infty} n/\sqrt d = \infty$, then
$$
\frac{\|S_n^{(d)}\|^2 - n}{ \sqrt{2n^2/d}} \todistrd \mathrm{N}(0,1).
$$
\end{enumerate}
\end{theorem}

\begin{rem}[Bibliographic comments]\label{rem:bibliographic}
High-dimensional asymptotic properties of trajectories of random walks have received, quite surprisingly, limited attention in the literature. For the simple symmetric random walk in $\R^d$ asymptotic behavior of the probability of returning to the origin, as $d\to\infty$, has been studied in~\cite{griffin1990accelerating,montroll1956random}. In the context of analysis and optimization of Metropolis-Hastings algorithms, infinite-dimensional diffusions pop up as the high-dimensional limits for the random-walk Metropolis algorithm; see~\cite[Theorem 13]{breyer2000metropolis} and also~\cite{roberts2001optimal}.

Explicit functionals of convex hulls of random walks (such as its volume, intrinsic volumes and the number of face) have been studied, for example, in~\cite{vysotsky_zaporozhets_convex_hulls_TAMS,wade2015convex,KVZ17b,KVZ17,mcredmond_wade}; see also references therein. The asymptotics of the expected number of faces as, both, the dimension $d$ and the number of steps $n$ of the random walk go to $\infty$ has been studied in~\cite{kabluchko2022lah}. A natural question in this context is whether there is certain ``functional limit theorem'' describing the limiting object of the convex hull of the random walk, as $n, d\to\infty$. For fixed $d$ (and a random walk with zero mean and finite second moment), the corresponding limiting object is the convex hull of a $d$-dimensional Brownian motion~\cite{wade2015convex}. 
Theorem~\ref{theo:conv_to_convex_hull_of_wiener_spiral_in_hilbert_up_to_isometry} gives an answer in the regime when $d,n\to\infty$.

The quantity $\|S_n^{(d)}\|^2$ is closely related to large random matrices. More precisely, if ${\bf X}_{d,n}$ is a $d\times n$ real matrix with columns $X_1^{(d)},\ldots,X_n^{(d)}$ and ${\bf 1}_n:=e_1+e_2+\cdots+e_n$, then
\begin{equation}\label{eq:matrix_representation}
\|S_n^{(d)}\|^2={\bf 1}_n^{\top}\cdot{\bf X}_{d,n}^{\top}{\bf X}_{d,n}\cdot{\bf 1}_n.
\end{equation}
In the setting of Model 1 with $\xi$ being a centered Gaussian random variable the random matrix ${\bf X}_{d,n}^{\top}{\bf X}_{d,n}$ (up to a deterministic multiplicative constant) is called the Wishart random matrix; see~\cite[Chapter 7]{pastur2011eigenvalue}. Representation~\eqref{eq:matrix_representation} suggests that at least some limit theorems for $\|S_n^{(d)}\|^2$ could be derived from the results on eigenvalues distribution of large random matrices. However, even in a simple case of the Wishart Ensemble this turns out to be a non-trivial task, since $\|S_n^{(d)}\|^2$ is not a {\it linear eigenvalue statistic} of ${\bf X}_{d,n}^{\top}{\bf X}_{d,n}$. In the setting of Model 2, central limit theorems for the quantity $\|S_n^{(d)}\|^2$ and its matrix-valued generalizations have been derived in~\cite[Theorem 1.1, Theorem 1.2]{grundmann2014limit}; see also~\cite{voit+1995,voit2009,rosler2011limit,Voit+2012+231+246} for high-dimensional CLT's for random walks on homogeneous spaces.
\end{rem}

\section{Proofs: Convergence to the Wiener Spiral} \label{sec:proof_LLN_wiener_spiral}

\subsection{Functional law of large numbers for the norm}\label{subsec:proof_LLN_norm_S_n}
We begin with a result whose proof contains the main idea of the proof of Theorem~\ref{theo:GH_conv_to_wiener_spiral}.
\begin{theorem}\label{theo:FWLLN}
Let $n=n(d)$ be an arbitrary sequence of positive integers such that $n(d)\to \infty$, as  $d\to\infty$. Under the assumptions (a)--(d) of Section~\ref{subsec:conv_to_wiener_spiral},
\begin{equation}\label{eq:weak_LLN_for_norm}
\sup_{t\in [0,1]}\left|\frac{\|S_{\lfloor nt\rfloor}^{(d)}\|^2}{n}-t\right|\topr 0,
\end{equation}
where $\toprnod$ denotes convergence in probability.
\end{theorem}
Before giving the proof of Theorem~\ref{theo:FWLLN} some preparatory work has to be done. First, observe that, for every $k\in\N_0$,
\begin{equation}\label{eq:decomposition}
\|S^{(d)}_{k}\|^2=\langle S^{(d)}_{k},S^{(d)}_{k}\rangle=\langle X_1^{(d)}+\dots+X^{(d)}_{k},X_1^{(d)}+\dots+X^{(d)}_{k}\rangle
=
T_k^{(d)} + Q_k^{(d)},
\end{equation}
where
\begin{equation}\label{eq:T_Q_decomposition}
T_k^{(d)} :=  \sum_{i=1}^{k}\|X_i^{(d)}\|^2,
\quad
Q_k^{(d)} := \sum_{\substack{i,j \in \{1,\dots, k\} \\ i\neq j}} \langle X_i^{(d)},X_j^{(d)}\rangle,
\quad
k\in \N,
\end{equation}
and $T_0^{(d)}:= 0$, $Q_0^{(d)}:=0$. Further, note that
\begin{equation}\label{eq:Y_i_nd_Q_nd_def}
Q_{k}^{(d)} = 2 \sum_{i=1}^k Y_i^{(d)},
\quad
Y_i^{(d)} := \langle X_i^{(d)}, S_{i-1}^{(d)}\rangle.
\end{equation}

It will be of major importance for what follows that $(Q_{n}^{(d)})_{n\in\mathbb{N}_0}$ is a martingale. More precisely, the following holds true.
\begin{lemma}\label{lem:martingale_diff}
For any $d$-dimensional random walk with i.i.d. zero-mean increments $X_{1}^{(d)}, \dots, X_{n}^{(d)}$, the random variables $Y_{1}^{(d)}, \dots, Y_{n}^{(d)}$  form a triangular array of martingale differences with respect to the natural filtration $\mathcal F_{1}^{(d)}  \subset \dots \subset \mathcal F_{n}^{(d)}$, where
$\mathcal F_{i}^{(d)}$ is the $\sigma$-algebra generated   by $X_{1}^{(d)}, \dots, X_{i}^{(d)}$, for all $i\in \{1,\dots, n\}$.
\end{lemma}
\begin{proof}
To prove the martingale difference property observe that $Y_{i}^{(d)}$ is $\mathcal F_i^{(d)}$-measurable and
$$
\E \left[Y_i^{(d)} \Big  | \mathcal F_{i-1}^{(d)}\right]
=
\E \left[\langle X_i^{(d)}, S_{i-1}^{(d)}\rangle \Big   | \mathcal F_{i-1}^{(d)}\right]
=
\sum_{j=1}^d  \E \left[X_{i,j}^{(d)} S_{i-1,j}^{(d)} \Big | \mathcal F_{i-1}^{(d)}\right]
= 0,
$$
for all $i=1,\dots, n$,  where we used that $S_{i-1,j}^{(d)}$ is $ \mathcal F_{i-1}^{(d)}$-measurable and that $X_{i,j}^{(d)}$ is independent of $\mathcal F_{i-1}^{(d)}$ and has zero mean.
\end{proof}

\begin{proof}[Proof of Theorem~\ref{theo:FWLLN}.]
To prove~\eqref{eq:weak_LLN_for_norm}, it suffices to show that
\begin{equation}\label{eq:fwlln_proof0}
\sup_{t\in [0,1]}\left|\frac{T_{\lfloor nt \rfloor}^{(d)}}{n}-t\right|~\topr~0,
\end{equation}
and
\begin{equation}\label{eq:fwlln_proof1}
\frac{\sup_{t\in [0,1]}|Q_{\lfloor nt \rfloor}^{(d)}|}{n}~\topr~0.
\end{equation}

\vspace*{2mm}
\noindent
\textit{Proof of~\eqref{eq:fwlln_proof0}.}
According to a version of the law of large numbers stated in Lemma~\ref{lem:wlln} in  Appendix~\ref{appA},
$$
f_d(t) :=  n^{-1} T_{\lfloor nt \rfloor}^{(d)} = \frac 1n \sum_{i=1}^{\lfloor nt \rfloor}\|X_i^{(d)}\|^2 \topr t,
$$
for every $t\geq 0$.
Since the functions $t\mapsto f_d(t)$ and $t\mapsto t$ are monotone in $t\in [0,1]$ and the latter function is continuous, this convergence in probability is in fact uniform by P\'olya's extension of Dini's theorem.
Indeed, for every $m\in \N$ the union bound yields
$$
\max_{i\in\{0,\dots,m\}} |f_d(i/m) - (i/m)| \topr 0.
$$
The monotonicity of $t\mapsto f_d(t)$ implies that
$$
\sup_{t\in [0,1]} |f_d(t) -t| \leq \max_{i\in\{0,\dots,m\}} |f_d(i/m) - (i/m)| + (1/m).
$$
Given $\eps>0$ we choose $m\in \N$ such that $1/m < \eps/2$. Then,
$$
\mathbb P\left[\sup_{t\in [0,1]} |f_d(t) -t|>\eps \right]
\leq
\mathbb P\left[\max_{i\in\{0,\dots,m\}} |f_d(i/m) - (i/m)| > \eps/2\right]
\overset{}{\underset{d\to\infty}\longrightarrow} 0.
$$
It follows that $\sup_{t\in [0,1]} |f_d(t) -t|$ converges in probability to $0$, thus proving~\eqref{eq:fwlln_proof0}.

\vspace*{2mm}
\noindent
\textit{Proof of~\eqref{eq:fwlln_proof1}.} Since $(Q_\ell^{(d)})_{\ell\in \N_0}$ is a martingale, for every fixed $d\in \N$, Doob's martingale inequality entails that
$$
\mathbb P\Big[\sup_{t\in [0,1]}|Q_{\lfloor nt \rfloor}^{(d)}|  \geq n \eps\Big] \leq \frac{\E (Q_n^{(d)})^2}{n^2\eps^2} .
$$
Hence, to prove~\eqref{eq:fwlln_proof1}, it suffices to check that
\begin{equation}\label{eq:fwlln_proof3}
\frac {\E (Q_n^{(d)})^2}{n^2}
=
\frac{1}{n^2}\E\Bigg(\sum_{\substack{i,j \in \{1,\dots, n\} \\ i\neq j}} \langle X_i^{(d)},X_j^{(d)}\rangle\Bigg)^2
=
\frac{1}{n^2}\E\Bigg(\sum_{\substack{i,j \in \{1,\dots, n\} \\ i\neq j}} \sum_{k=1}^{d}X_{i,k}^{(d)}X_{j,k}^{(d)}\Bigg)^2
~\toinfd~ 0.
\end{equation}
An alternative way to see this sufficiency is to apply Corollary 2 on p.~1888 in \cite{deLaPena:1992} with $\Phi(x)=x^2$ and $f_{i,j}(u,v)=\langle u,v\rangle$, leading to the estimate
$$
\E \left(\sup_{t\in [0,1]}|Q_{\lfloor nt \rfloor}^{(d)}|\right)^2=\E\Bigg(\max_{m\in\{1,\dots,n\}}\Bigg|\sum_{\substack{i,j \in \{1,\dots, m\} \\ i\neq j}} \langle X_i^{(d)},X_j^{(d)}\rangle\Bigg|\Bigg)^2
\leq
2048 \cdot \E\Bigg(\sum_{\substack{i,j \in \{1,\dots, n\} \\ i\neq j}} \langle X_i^{(d)},X_j^{(d)}\rangle\Bigg)^2.
$$
In order to prove~\eqref{eq:fwlln_proof3} we write
\begin{multline*}
\E\Bigg(\sum_{\substack{i,j \in \{1,\dots, n\} \\ i\neq j}} \sum_{k=1}^{d} X_{i,k}^{(d)}X_{j,k}^{(d)}\Bigg)^2
=
\sum_{k=1}^{d}\sum_{k'=1}^{d}\sum_{\substack{i,j \in \{1,\dots, n\} \\ i\neq j}}\sum_{\substack{i',j' \in \{1,\dots, n\} \\ i'\neq j'}}\E \left[X_{i,k}^{(d)}X_{i',k'}^{(d)}X_{j,k}^{(d)}X_{j',k'}^{(d)}\right]\\
=
2\sum_{k=1}^{d}\sum_{i\neq j}\E (X_{i,k}^{(d)})^2 \E (X_{j,k}^{(d)})^2=2n(n-1)\sum_{k=1}^{d} \E (X_{1,k}^{(d)})^2 \E(X_{1,k}^{(d)})^2,
\end{multline*}
where for the second equality we used that by independence, uncorrelatedness and $\E X_{i,k}^{(d)}=0$, the expectation $\E
[X_{i,k}^{(d)}X_{i',k'}^{(d)}X_{j,k}^{(d)}X_{j',k'}^{(d)}]$ vanishes unless $k=k'$ and $\{i,j\} = \{i', j'\}$.  It remains to note that
$$
\sum_{k=1}^{d} \E (X_{1,k}^{(d)})^2 \E(X_{1,k}^{(d)})^2\leq \left(\max_{k\in\{1,\dots,d\}}\E (X_{1,k}^{(d)})^2\right)\left(\sum_{j=1}^{d} \E (X_{1,j}^{(d)})^2\right)=\max_{k\in\{1,\dots,d\}}\E (X_{1,k}^{(d)})^2 ~\toinfd~ 0,
$$
where~\eqref{eq:assump_negligible} has been utilized on the last step. The proof of~\eqref{eq:fwlln_proof3} is complete.
\end{proof}

\subsection{Proof of Theorem~\ref{theo:GH_conv_to_wiener_spiral}}\label{subsec:proof_conv_to_wiener_spiral}
We identify the Wiener spiral $\mathbb W$ with the interval $[0,1]$ equipped with the metric $d(t,s)= \sqrt{|t-s|}$. Define a surjective map $\varphi_d:[0,1]\to \mathbb M_d$ by $\varphi_d(t) := S_{\lfloor nt\rfloor}^{(d)}/\sqrt n$. By Corollary 7.3.28 on page 258 of~\cite{Burago+Burago:2001}, the Gromov-Hausdorff distance between $\mathbb W$ and $\mathbb M_d$ is bounded above by twice the distortion of the map $\varphi_d$, that is
$$
d_{\text{GH}}(\mathbb W, \mathbb M_d) \leq
2
\sup_{0\leq s \leq t \leq 1} \left|\frac{\|S_{\lfloor nt\rfloor}^{(d)} - S_{\lfloor ns\rfloor}^{(d)}\|}{\sqrt n} - \sqrt {t-s} \right|.
$$
To prove the theorem, it suffices to verify that
$$
\sup_{0\leq s \leq t \leq 1} \left|\frac{\|S_{\lfloor nt\rfloor}^{(d)} - S_{\lfloor ns\rfloor}^{(d)}\|}{\sqrt n} - \sqrt {t-s} \right| \topr 0.
$$
Take some $m\in \N$. We know from Theorem~\ref{theo:FWLLN} that, for every $i=0,\dots, m-1$,
$$
\frac{\| S_{\lfloor(i/m)\cdot n\rfloor}^{(d)}\|}{\sqrt n} \topr \sqrt{\frac i m}.
$$
Moreover, for every integer $0\leq i \leq j \leq m$, by stationarity,
$$
\frac{\| S_{\lfloor(j/m)\cdot n\rfloor}^{(d)} - S_{\lfloor(i/m)\cdot n\rfloor}^{(d)} \|}{\sqrt n} \topr \sqrt{\frac {j-i} m}.
$$
By the union bound, it follows that, for every fixed $m\in \N$,
\begin{equation}\label{eq:GH_uniform_discretized}
\max_{0\leq i \leq j \leq m} \left|\frac{\| S_{\lfloor(j/m)\cdot n\rfloor}^{(d)} - S_{\lfloor(i/m)\cdot n\rfloor}^{(d)} \|}{\sqrt n} - \sqrt{\frac {j-i} m}\right| \topr 0.
\end{equation}
If $0\leq s \leq t \leq 1$ are such that $s\in [\frac im, \frac{i+1}{m})$ and $t\in [\frac jm, \frac{j+1}{m})$, then,  by the triangle inequality,
$$
\left|\frac{\|S_{\lfloor nt\rfloor}^{(d)} - S_{\lfloor ns\rfloor}^{(d)}\|}{\sqrt n} - \frac{\| S_{\lfloor(j/m)\cdot n\rfloor}^{(d)} - S_{\lfloor(i/m)\cdot n\rfloor}^{(d)} \|}{\sqrt n}\right|
\leq
\sup_{z\in [\frac im, \frac{i+1}{m}]} \frac{\|S_{\lfloor nz\rfloor}^{(d)} - S_{\lfloor n\cdot (i/m)\rfloor}^{(d)}\|}{\sqrt n}
+
\sup_{z\in [\frac jm, \frac{j+1}{m}]} \frac{\|S_{\lfloor nz\rfloor}^{(d)} - S_{\lfloor n\cdot (j/m)\rfloor}^{(d)}\|}{\sqrt n}.
$$
Consider the random variable
$$
\epsilon_{m,d} := \max_{i\in \{0,\dots, m-1\}} \sup_{z\in [\frac im, \frac{i+1}{m}]} \frac{\|S_{\lfloor nz\rfloor}^{(d)} - S_{\lfloor n\cdot (i/m)\rfloor}^{(d)}\|}{\sqrt n}.
$$
To complete the proof, it suffices to show that for every $\eps>0$,
\begin{equation}\label{eq:GH_double_limit_is_zero}
\lim_{m\to\infty} \limsup_{d\to\infty} \mathbb P[\epsilon_{m,d}\geq \eps] = 0.
\end{equation}
Applying the union bound and recalling that $X_1^{(d)}, \dots, X_n^{(d)}$  are i.i.d.\ we can write
$$
\mathbb P[\epsilon_{m,d}\geq \eps] \leq m \mathbb P\left[\frac 1n \sup_{t\in [0, \frac{1}{m}]} \|S_{\lfloor nt\rfloor}^{(d)}\|^2 \geq \eps^2\right].
$$
Recalling decompositions~\eqref{eq:decomposition} and~\eqref{eq:T_Q_decomposition}, observe that
$$
\|S_{\lfloor nt\rfloor}^{(d)}\|^2=T^{(d)}_{\lfloor nt\rfloor}+Q^{(d)}_{\lfloor nt\rfloor}.
$$
To complete the proof, it suffices to verify that
\begin{equation}\label{eq:need_to_prove1}
\lim_{m\to\infty} \limsup_{d\to\infty} m\, \mathbb P\left[T^{(d)}_{\lfloor n/m\rfloor} \geq  n\eps^2/2\right]=0,
\end{equation}
and
\begin{equation}\label{eq:need_to_prove2}
\lim_{m\to\infty} \limsup_{d\to\infty} m\, \mathbb P\left[\sup_{t\in [0, \frac{1}{m}]} Q^{(d)}_{\lfloor nt\rfloor}  \geq  n\eps^2/2\right]=0.
\end{equation}

\vspace*{2mm}
\noindent
\textit{Proof of~\eqref{eq:need_to_prove1}.}
We observe that, for every fixed $m\in \N$,  $T^{(d)}_{\lfloor n/m\rfloor}/n$ converges in probability to $1/m$ by the version of the law of large numbers stated in Lemma~\ref{lem:wlln}. This implies that for every $m>2/\eps^2$, the $\limsup_{d\to\infty}$ in~\eqref{eq:need_to_prove1} equals $0$.

\vspace*{2mm}
\noindent
\textit{Proof of~\eqref{eq:need_to_prove2}.} By yet another appeal to Doob's martingale inequality we obtain
$$
m\, \mathbb P\left[\sup_{t\in [0, \frac{1}{m}]} Q^{(d)}_{\lfloor nt\rfloor}  \geq  n\eps^2/2\right]
\leq
\frac{m}{n^2 (\eps^2/2)^2} \E (Q^{(d)}_{\lfloor n/m\rfloor})^2.
$$
As we have already shown in~\eqref{eq:fwlln_proof3}, for every $m\in \N$,
$$
\lim_{d\to\infty} \frac {\E (Q^{(d)}_{\lfloor n/m\rfloor})^2 }{n^2} = 0.
$$
It follows that the $\limsup_{d\to\infty}$ in~\eqref{eq:need_to_prove2} equals $0$ for every $m\in \N$.
\hfill $\Box$

\begin{proof}[Proof of Corollary~\ref{cor:linear_interpolation}]
Note that $\mathbb{M}_d \subset \mathbb{M}^{\mathrm{cont}}_d$ and
$$
d_{GH}(\mathbb{M}^{\mathrm{cont}}_d,\mathbb{M}_d)\leq d_{H}(\mathbb{M}^{\mathrm{cont}}_d,\mathbb{M}_d)\leq \max_{i\in \{1,\ldots,n\}}\frac{\|X^{(d)}_i\|}{\sqrt{n}}.
$$
The right-hand side converges to zero in probability, since, for every fixed $\varepsilon>0$,
$$
\mathbb{P}\left[ \sup_{i\in \{1,\ldots,n\}}\frac{\|X^{(d)}_i\|}{\sqrt{n}} > \varepsilon \right]\leq n\mathbb{P}[\|X^{(d)}\|^2 > \varepsilon^2 n]
\leq \varepsilon^{-2} \mathbb{E} [\|X^{(d)}\|^2 \ind_{\{\|X^{(d)}\|^2 > \varepsilon^2 n\}}]
\leq \varepsilon^{-2} \sup_{\ell\in\mathbb{N}}\mathbb{E} [\|X^{(\ell)}\|^2 \ind_{\{\|X^{(\ell)}\|^2 > \varepsilon^2 n\}}],
$$
and the latter converges to zero by~\eqref{eq:assump_uniform_integrability}.
\end{proof}

\subsection{Proof of Theorem~\ref{theo:iid_components_wiener_spiral}}
The map $K\ni t\mapsto \mathbb{X}_d(t)\in \mathbb{K}_d$ is surjective. Similarly to the proof of Theorem~\ref{theo:GH_conv_to_wiener_spiral} we use Corollary 7.3.28 on page 258 of~\cite{Burago+Burago:2001} to infer that
$$
d_{\text{GH}}(\mathbb K_d, K) \leq 2  \sup_{s,t \in K} \left|\|\mathbb X_d(s) - \mathbb X_d(t)\| - \rho(s,t)\right|.
$$
To prove the theorem it suffices to show that the right-hand side converges to $0$ a.s., that is
\begin{equation}\label{eq:distortion_to_zero}
\sup_{s,t \in K} \left|\|\mathbb X_d(s) - \mathbb X_d(t)\| - \sqrt{\Var (X(s) - X(t))}\right| \toasd 0.
\end{equation}
The function $z\mapsto \sqrt z$ is uniformly continuous on every interval of the form $[0,A]$, with $A>0$.  Therefore, for non-negative bounded functions, $f_n\to f$ uniformly implies that $\sqrt f_n \to \sqrt f$ uniformly.  Hence, to prove~\eqref{eq:distortion_to_zero}, it suffices to check that
\begin{equation}\label{eq:distortion_to_zero_squared}
\lim_{d\to\infty} \sup_{s,t \in K} \left|\|\mathbb X_d(s) - \mathbb X_d(t)\|^2 - \Var (X(s) - X(t)) \right| = 0.
\end{equation}
Define i.i.d.\ stochastic processes $(Y_k(s,t))_{(s,t)\in K\times K}$, $k\in \N$,  by
$$
Y_k(s,t) = (X_k(s) - X_k(t))^2 - \E (X_k(s) - X_k(t))^2, \quad (s,t)\in K\times K.
$$
Note that $Y_k$ has continuous sample paths on $K\times K$ (endowed with the product metric~\cite[p.~88]{Burago+Burago:2001}) and that
$$
\E \left[\sup_{(s,t)\in K\times K} |Y_k(s,t)|\right] \leq 2 \E \left[\sup_{(s,t)\in K\times K} (X_k(s) - X_k(t))^2\right]
 \leq  4 \E \left[\sup_{(s,t)\in K\times K} (X_k^2(s) + X_k^2(t))\right] \leq 8\E \left[\sup_{s\in K} X^2_k(s)\right] < \infty.
$$
Then,
$$
\|\mathbb X_d(s) - \mathbb X_d(t)\|^2 - \Var (X(s) - X(t)) =  \frac 1d \sum_{k=1}^d ((X_k(s) - X_k(t))^2 - \E (X_k(s) - X_k(t))^2)
=
\frac 1d \sum_{k=1}^d Y_k(s,t).
$$
Note that  $Y_1,Y_2,\dots$ are i.i.d.\ random elements in the Banach space $C(K\times K)$ of continuous functions on the compact space $K\times K$. As we have shown,  $\E \|Y_k\|_\infty < \infty$.    By the strong law of large numbers in the Banach space $C(K\times K)$, see Theorem~1.1 on page 131 in~\cite{hoffmann_joergensen_LNM598}, we have
$$
\sup_{(s,t)\in K\times K } \left|\frac 1d \sum_{k=1}^d Y_k(s,t)\right| \toasd  0.
$$
This proves~\eqref{eq:distortion_to_zero_squared} and completes the proof of Theorem~\ref{theo:iid_components_wiener_spiral}.
\hfill $\Box$

\subsection{Proofs for Section~\ref{subsec:hausdorff_up_to_isometry}}\label{sec:proof_up_to_isometry} 

In this section we collect the proofs related to convergence in $(\bH,d_{\sim})$ with the only exception of Lemma~\ref{lem:isometry_extension_to_affine_hull} whose proof is given in the Appendix~\ref{appD}.

\begin{proof}[Proof of Proposition~\ref{prop:tilde_d_is_a_metric}]
It is clear that $d_{\sim}([K],[K]) = 0$ and $d_{\sim}([K_1], [K_2]) = d_{\sim}([K_2], [K_1])$. Let us check that $d_{\sim}([K_1], [K_2]) = 0$ implies $K_1\sim K_2$. Indeed, $d_{\sim}([K_1], [K_2]) = 0$ implies that the usual Gromov-Hausdorff distance between $K_1$ and $K_2$ (both endowed with the metric induced from $\ell^2$) is $0$. This implies that $K_1$ is isometric to $K_2$; see~\cite[Theorem 7.3.30]{Burago+Burago:2001}, that is, there is a distance-preserving bijection $J: K_1 \to K_2$, proving that $K_1\sim K_2$. 

Let us prove the triangle inequality. Take $[K_1], [K_2], [K_3] \in \bH$ and put $d_{ij}:= d_{\sim} ([K_i], [K_j])$ for $i,j \in \{1,2,3\}$. Our aim is to prove that $d_{13} \leq d_{12} + d_{23}$. Fix $\eps>0$. By the definition of $d_{\sim}$, there exist representatives $K_1'\in [K_1]$ and $K_2'\in [K_2]$ with $d_{H}(K_1', K_2') \leq d_{12} + \eps$. Similarly, there exist representatives $K_2''\in [K_2]$ and $K_3''\in [K_3]$ with $d_{H}(K_2'', K_3'') \leq d_{23} + \eps$. The problem is that, unfortunately, $K_2'$ need not be the same as $K_2''$.

Without loss of generality we assume that $\codim \overline{\aff} K_2' = \infty$ and $\codim \overline{\aff} K_2'' = \infty$. Indeed, otherwise we let $L$ be the closed linear hull of the basis vectors $e_2,e_4,e_6,\ldots$ and $S: \ell^2 \to L$ be a linear isometry defined by $S(e_k)=e_{2k}$, $k\in\mathbb{N}$. Then, $SK_1' \in [K_1]$, $SK_2' \in [K_2]$ and $d_{H}(S K_1', S K_2') = d_{H}(K_1', K_2') \leq d_{12} + \eps$. Thus, we can replace $K_1'$ and $K_2'$ by $S K_1'$ and $S K_2'$ which are contained in the closed linear subspace $L$ satisfying $\codim L = \infty$. Similar argument can be applied to $K_2''$ and $K_3''$.

By Lemma~\ref{lem:isometry_extension_to_affine_hull}, the isometry $J: K_2'' \to K_2'$ (which exists since $K_2', K_2''\in [K_2]$) can be extended to an isometry between $L_2'':= \overline{\aff} K_2''$ and $L_2':=\overline{\aff} K_2'$. As argued above, we can assume that $\codim L_2' = \codim L_2'' = \infty$. We now claim that, in fact, we can extend $J$ to a self-isometry of the whole $\ell^2$.  More precisely, after a shift we may assume that $L_2'$ and $L_2''$ are closed \textit{linear} subspaces. Let $I: (L_2'')^{\bot} \to (L_2')^{\bot}$ be any isometry between the orthogonal complements of $L_2''$ and $L_2'$ (which exists since both complements are separable, infinite-dimensional Hilbert spaces).  Then, we can extend the isometry $J : L_2'' \to L_2'$ to $J : \ell^2\to\ell^2$ by putting $J(x + y) = J(x) + I(y)$ for any $x\in L_2''$ and $y\in (L_2'')^{\bot}$.

Now, we observe that $d_{H}(K_1', K_2') \leq d_{12} + \eps$ and $d_{H}(JK_2'', JK_3'') = d_{H}(K_2'', K_3'') \leq d_{23} + \eps$. Recall that $JK_2'' = K_2'$. The triangle inequality for the metric $d_{H}$ yields
$$
d_{H} (K_1', J K_3'') \leq d_{H}(K_1', K_2') + d_{H}(JK_2'', JK_3'')
\leq
d_{12} + d_{23} + 2\eps.
$$
Since $\eps>0$ is arbitrary and $K_1'\in [K_1]$, $J K_3'' \in [K_3]$, this proves that $d_{\sim}([K_1], [K_3]) \leq d_{12} + d_{23}$.
\end{proof}

\begin{proof}[Proof of Theorem~\ref{theo:equivalence_GH_and_up_to_iso}]

By Remark~\ref{rem:GH_metric_vs_Hilbert_GH} convergence in $(\bH, d_{\sim})$ implies convergence in the Gromov-Hausdorff sense. Let us prove the inverse implication.
Assume that $K_n\to K$ (as $n\to\infty$) in the Gromov-Hausdorff sense. It follows that the diameters of $K_n$ and $K$ are uniformly bounded by some $R>0$, meaning that $\|x-x'\|\leq R$ and $\|y-y'\|\leq R$ for all $x, x'\in K$ and all $n\in \N$, $y, y'\in K_n$. Our aim is to show that $[K_n] \to [K]$ in $(\bH, d_{\sim})$, as $n\to\infty$.   Fix some $\eps>0$. Let $x_0,x_1,\ldots, x_m$ be an $\eps$-net in $K$ meaning that for every $x\in K$ there exists $i\in \{0,\ldots, m\}$ with $\|x - x_i\| <\eps$.
Take some $\delta >0$. There is an $n_0= n_0(\delta)$ such that for all $n>n_0$ we have $d_{GH}(K_n, K) <\delta$. Fix some $n>n_0$.  By definition of the Gromov-Hausdorff distance, there is a metric space $(M,\rho)$ and isometric embeddings $\varphi: K \to M$ and $\psi: K_n \to M$ such that $d_H(\varphi(K), \psi(K_n)) < \delta$. It follows that for each $i\in \{0,\ldots, m\}$ there is a point $y_i \in K_n$ such that $\rho(\varphi(x_i), \psi(y_i)) <\delta$. Using the triangle inequality, it is easy to check that the points $y_0,y_1,\ldots, y_m$ form a $(\eps + 2\delta)$-net in $K_n$. Using again the triangle inequality, we obtain
$$
\|x_i -x_j\| = \rho(\varphi(x_i), \varphi(x_j)) \leq \rho(\psi(y_i), \psi(y_j)) + 2\delta = \|y_i -y_j\| + 2\delta,
$$
and also
$$
\|y_i -y_j\| = \rho(\psi(y_i), \psi(y_j)) \leq \rho(\varphi(x_i), \varphi(x_j)) + 2\delta = \|x_i -x_j\| + 2\delta.
$$
Taking the squares and recalling that the diameters of all $K_n$ and $K$ are bounded above by $R$, yields
$$
\left|\|x_i -x_j\|^2 - \|y_i-y_j\|^2 \right|\leq  4\delta R + 4\delta^2,
\qquad i,j \in \{0,\ldots, m\}.
$$
Applying suitable shifts, we may assume that $x_0 = y_0 = 0$. Using the parallelogram law
$$
\langle u, v\rangle=\frac{1}{2}\left(\|u\|^2+\|v\|^2-\|u-v\|^2\right),\quad u,v\in\ell^2,
$$
we conclude that, for all $i,j\in \{1,\ldots, m\}$,
$$
\left|\langle x_i, x_j\rangle -\langle y_i, y_j\rangle \right|\leq 6\delta R + 6 \delta^2=: \eta.
$$
Let $G_x = (\langle x_i, x_j\rangle)_{i,j=1}^m$ (respectively, $G_y = (\langle y_i, y_j\rangle)_{i,j=1}^m$)  be the Gram matrix of the vectors $x_1,\ldots, x_m$  (respectively, $y_1,\ldots, y_m$). Let $\cP_m$ be the set of positive semidefinite $m\times m$-matrices, so that $G_x, G_y\in \cP_m$. We endow the space of $m\times m$-matrices with the norm $\|A\|_\infty = \max_{i,j\in\{1,\ldots, m\}}|A_{ij}|$ and the set $\cP_m$ with the metric induced by this norm. Then, as we have shown above, $\|G_x - G_y\|_\infty \leq \eta$. Since the mapping $G\mapsto G^{1/2}$ is continuous on $\cP_m$, we see that
$$
\|G_x^{1/2} - G_y^{1/2}\|_\infty \leq f_m(\eta),
$$
for some function $f_m(\eta)$ such that $\lim_{\eta\downarrow 0} f_m(\eta) = 0$.

The Gram matrix of the vectors $G_x^{1/2}e_1,\ldots, G_x^{1/2}e_m$ in $\R^m\subset \ell^2$ is $G_x$ and coincides with the Gram matrix of the vectors $x_1,\ldots, x_m$ in $\ell^2$. It follows that there is an orthogonal transformation $O:\ell^2 \to \ell^2$  such that $O G_x^{1/2}e_i = x_i$, for all $1\leq i \leq m$. Applying the same orthogonal transformation $O$ to the vectors $G_y^{1/2}e_1,\ldots, G_y^{1/2}e_m$ we define
\begin{equation}\label{eq:W_i_m_tech_def_wiener_spiral_skel}
\widetilde{y}_i := O G_y^{1/2} e_i \subset \ell^2, \quad 1\leq i \leq m, \qquad \widetilde y_0:=0.
\end{equation}
Then, for all $i\in \{1,\ldots, m\}$
$$
\left\|x_i - \widetilde{y}_{i}\right\|
=
\left\|O G_x^{1/2}e_i -O G_y^{1/2}e_i \right\|
=
\left\|G_x^{1/2}e_i - G_y^{1/2}e_i \right\|
\leq
m \|G_x^{1/2} - G_y^{1/2}\|_\infty
\leq
mf_m(\eta).
$$
By definition, the Gram matrix of $\widetilde {y_1},\ldots, \widetilde {y_m}$ is the same as of $y_1,\ldots, y_m$. Note that both systems of vectors span linear subspaces of infinite codimension in $\ell^2$. Thus, the isometry sending $0$ to $0$ and $y_i$ to $\widetilde y_i$ for all $i\in \{1,\ldots, m\}$ can be extended to a global self-isometry $U$ of $\ell^2$ by Lemma~\ref{lem:isometry_extension_to_affine_hull} and the last remark in Example~\ref{example:two_isometric_sets_no_global_isometry}; see also the proof of Proposition~\ref{prop:tilde_d_is_a_metric}. We claim that the Hausdorff distance between $UK_n$ and $K$ satisfies
$$
d_H(U K_n, K) \leq m f_m(6\delta R + 6 \delta^2) + \eps + 2\delta.
$$
Indeed, for every point $\widetilde y\in UK_n$ is at distance $<\eps + 2\delta$ from some $\widetilde y_i$ which is a distance $< m f_m(\eta)$ from $x_i$. Conversely, every $x\in K$ is at distance $<\eps$ from some $x_i$ which is a at distance $<m f_m(\eta)$ from  $\widetilde {y_i}$.

Now, for a fixed $\eps>0$ (and the corresponding $m=m(\eps)$) we choose $\delta= \delta(\eps, m)>0$ to ensure that $m f_m(6\delta R + 6 \delta^2) + \eps + 2\delta < 3\eps$. The above shows that for all $n>n_0(\delta)$, we have $d_{\sim} ([K_n], [K]) \leq 3\eps$. Since $\eps>0$ is arbitrary, this proves that  $[K_n] \to [K]$ in $(\bH, d_{\sim})$.
\end{proof}

\begin{proof}[Proof of Theorem~\ref{theo:conv_to_convex_hull_of_wiener_spiral_in_hilbert_up_to_isometry}]
The proof follows from Theorem~\ref{theo:conv_to_wiener_spiral_in_hilbert_up_to_isometry} in view of the continuous mapping theorem whose use is justified by the next lemma.
\end{proof}

\begin{lemma}\label{lem:closed_convex_hull_Lip}
The closed convex hull map $\mathrm{CCH}: \bH \to \bH$ given by $\mathrm{CCH}([K]) = [\overline{\conv} K]$ is well defined and $1$-Lipschitz, that is,
$$
d_{\sim}(\mathrm{CCH}([K_1]), \mathrm{CCH}([K_2])) \leq d_{\sim}([K_1], [K_2])
$$
for all $[K_1], [K_2] \in \bH$.  In particular, the map $\mathrm{CCH}$ is continuous.
\end{lemma}
\begin{proof}
Let us first verify that the map is well-defined meaning that
\begin{equation}\label{eq:convex_hull_well_defined}
K_1\sim K_1^{\prime}\quad \Longrightarrow\quad  \overline{\conv} K_1\sim \overline{\conv} K_1^{\prime}
\end{equation}
and, therefore, $[\overline{\conv} K_1] = [\overline{\conv} K_1^{\prime}]$. By Lemma~\ref{lem:isometry_extension_to_affine_hull} there is an affine, isometric map $J:\overline{\aff} K_1 \to \overline{\aff} K_1^{\prime}$ such that $J(K_1) = K_1^{\prime}$. Since $J$ is affine, we have $J(\conv K_1) = \conv K_1^{\prime}$. From the isometric property of $J$ it follows that $J(\overline{\conv} K_1) = \overline{\conv}K_1^{\prime}$, which proves the claim.

To prove the $1$-Lipschitz property, suppose that $K^{\prime}_1\in [K_1]$ and $K^{\prime}_2\in [K_2]$ are such that $d_{H}(K_1^{\prime},K_2^{\prime}) = r$. Fix $\eps>0$. It follows that $K_1^{\prime} \subset K_2^{\prime} + B_{r+\eps}$ and $K_2^{\prime}\subset K_1^{\prime} + B_{r+\eps}$, where $B_{r+\eps}$ is the ball (in $\ell^2$) of radius $r+\eps$ centered at the origin and $+$ denotes the Minkowski addition. One checks directly that $\conv (C_1 + C_2) \subset \conv C_1 +  \conv C_2$ for arbitrary sets $C_1,C_2\subset \ell^2$. Thus,
$$
\overline{\conv K_1^{\prime}} \subset \overline{\conv(K_2^{\prime} + B_{r+\eps})} \subset \overline{\conv (K_2^{\prime}) + B_{r+\eps}} \subset \conv(K_2^{\prime}) + B_{r+2\eps},
$$
where the last step follows from $\bar C \subset C+ B_\eps$ for every $\eps>0$. Similarly, one shows that $\overline{\conv K_2^{\prime}} \subset \conv (K_1^{\prime}) + B_{r+2\eps}$. By definition of the Hausdorff distance, $d_{H}(\overline{\conv} K_1^{\prime},\overline{\conv} K_2^{\prime}) \leq r$. Thus,
$$
d_{H}(\overline{\conv}K_1^{\prime},\overline{\conv} K_2^{\prime})\leq d_{H}(K_1^{\prime},K_2^{\prime}).
$$
Passing to infimums yields
\begin{multline*}
d_{\sim}([K_1], [K_2])=\inf_{K_1'\in [K_1], K_2'\in [K_2]} d_{H}(K_1^{\prime},K_2^{\prime})\geq \inf_{K_1'\in [K_1], K_2'\in [K_2]}d_{H}(\overline{\conv} K_1^{\prime},\overline{\conv} K_2^{\prime})\\
\geq
d_{\sim}(\mathrm{CCH}([K_1]), \mathrm{CCH}([K_2])),
\end{multline*}
where for the last passage we used that infimum in the definition of $d_{\sim}(\mathrm{CCH}([K_1]), \mathrm{CCH}([K_2]))$ is taken over a larger set by~\eqref{eq:convex_hull_well_defined}.
\end{proof}

\section{Proofs: The distributional limit theorems for the squared norm}\label{sec:proofs_CLT}
\subsection{General strategy}\label{subsec:general_strategy}
To prove the results stated in Section~\ref{sec:results_CLT}, recall from~\eqref{eq:decomposition}, \eqref{eq:T_Q_decomposition}, \eqref{eq:Y_i_nd_Q_nd_def} the decomposition
\begin{equation*}
\|S^{(d)}_{n}\|^2 - n
=
\langle S^{(d)}_{n},S^{(d)}_{n}\rangle - n
=
\langle X_1^{(d)}+\dots+X^{(d)}_{n},X_1^{(d)}+\dots+X^{(d)}_{n}\rangle - n
=
T_n^{(d)} - n  + Q_n^{(d)}.
\end{equation*}
Our aim is to derive distributional limit theorems for the ``diagonal sum'' $T_n^{(d)}$ and the ``off-diagonal sum'' $Q_n^{(d)}$. For the former quantity, this task is usually straightforward since $T_{n}^{(d)}$ is a sum of i.i.d.\ random variables. Suppose that
\begin{equation}\label{eq:distr_limit_T}
\frac{T_n^{(d)} - n}{\tau_n^{(d)}}\todistrd  T_\infty
\end{equation}
for a suitable normalizing sequence $\tau_n^{(d)}>0$ and some stable random variable $T_\infty$.
For the off-diagonal sum, we shall prove, in all three models, a central limit theorem of the form
\begin{equation}\label{eq:distr_limit_R}
\frac{Q_n^{(d)}}{\sqrt{2n^2/d}} \todistrd \mathrm{N}(0,1).
\end{equation}
Having~\eqref{eq:distr_limit_T} and~\eqref{eq:distr_limit_R} at our disposal, we can determine the limit distribution of $\|S^{(d)}_{n}\|^2 - n$.
Depending on which of the normalizing sequences, $\tau_n^{(d)}$ or $\sqrt{2n^2/d}$, is asymptotically larger, we distinguish the following cases.

\noindent
\textsc{Case 1}: Off-diagonal fluctuations dominate meaning that $\tau_n^{(d)} = o(\sqrt{n^2/d})$. Then,
$$
\frac{\|S^{(d)}_{n}\|^2 - n}{\sqrt{2n^2/d}}
=
\frac{T_n^{(d)} - n}{\tau_n^{(d)}} \cdot \frac{\tau_n^{(d)}}{\sqrt{2n^2/d}} +  \frac{Q_n^{(d)}}{\sqrt{2n^2/d}}
\todistrd
\mathrm{N}(0,1).
$$
\noindent
\textsc{Case 2}: Diagonal fluctuations dominate meaning that $\sqrt{n^2/d} = o(\tau_n^{(d)})$. Then,
$$
\frac{\|S^{(d)}_{n}\|^2 - n}{\sqrt{2n^2/d}}
=
\frac{T_n^{(d)} - n}{\tau_n^{(d)}}  +  \frac{Q_n^{(d)}}{\sqrt{2n^2/d}} \cdot \frac{\sqrt{2n^2/d}}{\tau_n^{(d)}}
\todistrd
T_\infty.
$$

\noindent
\textsc{Case 3:} Both types of fluctuations are of the same order meaning that $\sqrt{2n^2/d} / \tau_n^{(d)} \to c\in (0,\infty)$. This case is somewhat more difficult and requires a separate analysis.

\subsection{Central limit theorem for the off-diagonal sum}\label{subsec:CLT_off_diagonal_sum}
In all three models, the proof of the CLT for $Q_{n}^{(d)}$ is based on the representation~\eqref{eq:Y_i_nd_Q_nd_def}.

To prove a central limit theorem for $Q_n^{(d)}$ we are going to apply the martingale central limit theorem, see Theorem~\ref{theo:martingale_CLT} in Appendix~\ref{appB}, to the martingale differences
$$
\Delta_i^{(d)} := \frac{Y_i^{(d)}}{\sqrt{n^2/(2d)}}, \quad i=1,\dots, n,
$$
where $Y_i^{(d)}$ was defined in~\eqref{eq:Y_i_nd_Q_nd_def}. If the conditions of Theorem~\ref{theo:martingale_CLT} are satisfied with $\sigma^2=1$, then
$$
\frac{Q_{n}^{(d)}}{\sqrt{2n^2/d}} = \Delta_1^{(d)} + \dots + \Delta_n^{(d)} \todistrd \mathrm{N} (0,1).
$$
In the following two lemmas we simultaneously verify condition~\eqref{eq:martingale_CLT_cond_1} of Theorem~\ref{theo:martingale_CLT} for all three models defined in Sections~\ref{subsec:model_1}, \ref{subsec:model_2}, \ref{subsec:model_3}.
\begin{lemma}\label{lem:V_n_d_and_Var_Q_n_d }
Consider a $d$-dimensional random walk $(S_i^{(d)})_{i=0}^\infty$ with i.i.d. zero-mean increments $X_{1}^{(d)},X_{2}^{(d)},\dots$ satisfying
\begin{equation}\label{eq:assump_uncor_equal_var}
\E [X_{i,j}^{(d)}X_{i,k}^{(d)}] =0,
\quad
\E (X_{i,j}^{(d)})^2 =1/d,
\quad
j,k \in \{1,\dots, d\},
\quad
j\neq k,
\quad
i\in\N.
\end{equation}
Then, for all $n\in \N$,
\begin{equation}\label{eq:V_n_d}
D_n^{(d)}
:=
\sum_{i=1}^n \E [ (Y_i^{(d)})^2 |\mathcal F_{i-1}^{(d)}]
=
\frac 1 d \sum_{i=1}^{n-1} \|S_i^{(d)}\|^2,
\end{equation}
and
\begin{equation}\label{eq:V_n_d_Q_n_d_expect}
\E D_n^{(d)} = \sum_{i=1}^n \E (Y_i^{(d)})^2  = \frac{n(n-1)}{2d},
\quad
\Var Q_n^{(d)} = 4\E D_n^{(d)} = \frac{2n(n-1)}{d}.
\end{equation}
\end{lemma}
\begin{proof}
To prove~\eqref{eq:V_n_d}, observe that
\begin{align*}
D_n^{(d)}
&=
\sum_{i=1}^n \E \left[ \langle X_i^{(d)},S_{i-1}^{(d)}\rangle^2 \Big |\mathcal F_{i-1}^{(d)}\right]
=
\sum_{i=1}^n \E \left[ \left(\sum_{j=1}^d X_{i,j}^{(d)} S_{i-1, j}^{(d)}\right)^2  \Big |\mathcal F_{i-1}^{(d)}\right]\\
&=
\sum_{i=1}^n \E \left[ \sum_{k=1}^d \sum_{\ell=1}^d  X_{i,k}^{(d)}X_{i,\ell}^{(d)} S_{i-1, k}^{(d)} S_{i-1, \ell}^{(d)} \Big |\mathcal F_{i-1}^{(d)}\right]
=
\sum_{i=1}^n \sum_{k=1}^d \sum_{\ell=1}^d S_{i-1, k}^{(d)} S_{i-1, \ell}^{(d)} \E \left[X_{i,k}^{(d)}X_{i,\ell}^{(d)}\right]\\
&=
\sum_{i=1}^n \sum_{j=1}^d  (S_{i-1, j}^{(d)})^2  \E (X_{i,j}^{(d)})^2
=
\frac 1d \sum_{i=1}^n \sum_{j=1}^d  (S_{i-1, j}^{(d)})^2
=
\frac 1d \sum_{i=1}^{n-1} \|S_{i, j}^{(d)}\|^2,
\end{align*}
where we used~\eqref{eq:assump_uncor_equal_var}. To prove the first equation in~\eqref{eq:V_n_d_Q_n_d_expect}, take expectation of~\eqref{eq:V_n_d} and observe that $\E \|S_i^{(d)}\|^2 = i$. To prove the second equation in~\eqref{eq:V_n_d_Q_n_d_expect}, recall~\eqref{eq:Y_i_nd_Q_nd_def} and observe that $Y_1^{(d)},\dots, Y_n^{(d)}$, being martingale differences, see Lemma~\ref{lem:martingale_diff}, are uncorrelated.
\end{proof}

\begin{lemma}\label{lem:cond_1_mart_CLT_verification}
For a random walk satisfying conditions (a), (c), (d) of Section~\ref{subsec:conv_to_wiener_spiral} and~\eqref{eq:assump_uncor_equal_var} we have
$$
\sum_{i=1}^n \E [ (\Delta_i^{(d)})^2 |\mathcal F_{i-1}^{(d)}] = \frac{D_n^{(d)}}{n^2/(2d)} \topr 1.
$$
\end{lemma}
\begin{proof}
We know from Theorem~\ref{theo:FWLLN} that
$$
\max_{i\in \{1,\ldots,n\}} \frac 1n \left|\|S_i^{(d)}\|^2 - i\right|  \topr 0.
$$
Taking some $\eps>0$ and denoting by $A_n^{(d)}$ the event that $\max_{i\in \{1,\ldots,n\}} |\|S_i^{(d)}\|^2 - i| \leq  n \eps$, we have that $\mathbb P[A_n^{(d)}]\to 1$ as $d\to\infty$. On the event $A_n^{(d)}$ we have the upper bound
$$
D_n^{(d)} = \frac 1 d \sum_{i=1}^{n-1} \|S_i^{(d)}\|^2 \leq \frac 1d \sum_{i=1}^{n-1} (i + n\eps) = \frac{n(n-1) + 2\eps n(n-1)}{2d}
$$
and the lower bound
$$
D_n^{(d)} = \frac 1 d \sum_{i=1}^{n-1} \|S_i^{(d)}\|^2 \geq \frac 1d \sum_{i=1}^{n-1} (i - n\eps) = \frac{n(n-1) - 2\eps n(n-1)}{2d}.
$$
Taken together, these bounds imply the claim.
\end{proof}

\subsection{Model 1: Proofs of Theorems~\ref{theo:distr_norm_model_1} and~\ref{theo:distr_norm_model_1_stable}}
The main difficulty is to prove the following central limit theorem for $Q_n^{(d)}$.
\begin{prop}\label{prop:model_1_off_diag_CLT}
In the setting of Section~\ref{subsec:model_1} suppose that $\E |\xi|^{2+ \delta} <\infty$, for some $\delta>0$. Then,
$$
\frac{Q_n^{(d)}}{\sqrt{2n^2/d}} \todistrd \mathrm{N}(0,1).
$$
\end{prop}
\begin{proof}
Since condition~\eqref{eq:martingale_CLT_cond_1} of the martingale central limit theorem (see Theorem~\ref{theo:martingale_CLT}) has already been verified in Lemma~\ref{lem:cond_1_mart_CLT_verification}, it remains to verify  Lyapunov's condition~\eqref{eq:martingale_CLT_cond_lyapunov_2} which takes the form
$$
\sum_{i=1}^n \E |Y_i^{(d)}|^{2+\delta} = o(n^{2+\delta}/ d^{1 + \frac \delta 2}),
$$
where $\delta>0$ is such that $\E |\xi|^{2+\delta}<\infty$. To prove this estimate, it suffices to show that
\begin{equation}\label{eq:referee1}
\max_{i\in \{1,\ldots,n\}} \E |Y_i^{(d)}|^{2+\delta} \leq C (n/d)^{1 + \frac \delta 2}.
\end{equation}
In the following, $C$ denotes a sufficiently large constant that does not depend on $d$. Recall from Section~\ref{subsec:model_1} that
$$
Y_i^{(d)} = \langle X_i^{(d)}, S_{i-1}^{(d)}\rangle = \frac 1d \sum_{j=1}^d \xi_{i,j} (\xi_{1,j} + \dots + \xi_{i-1, j}) =: \frac 1d \sum_{j=1}^d \xi_{i,j} \eta_{i-1, j},
$$
where we defined $\eta_{i-1,j} := \xi_{1,j} + \dots + \xi_{i-1, j}$. The Rosenthal inequality; see Theorem~\ref{theo:rosenthal_ineq} in Appendix~\ref{appC}, implies that
\begin{equation}\label{eq:referee2}
\E |Y_i^{(d)}|^{2 + \delta} = d^{-2-\delta} \E \left |\sum_{j=1}^d \xi_{i,j} \eta_{i-1, j}\right|^{2 + \delta}
\leq
C d^{-2-\delta} \max\left\{\sum_{j=1}^d \E |\xi_{i,j} \eta_{i-1, j}|^{2+\delta}, \left(\sum_{j=1}^d \E (\xi_{i,j}\eta_{i-1,j})^2\right)^{1 + \frac \delta 2} \right\},
\end{equation}
for all $i\in \{1,\dots, n\}$.
In the following, we estimate both terms appearing on the right-hand side.
For the first term, we first recall that $\xi_{i,j}$ is independent of $\eta_{i-1, j}$ and has finite moment of order $(2+\delta)$:
$$
\sum_{j=1}^d \E |\xi_{i,j} \eta_{i-1, j}|^{2+\delta}
\leq C \sum_{j=1}^d \E |\eta_{i-1, j}|^{2+\delta}, \quad i\leq n.
$$
For each summand on the right-hand side we use the Rosenthal inequality to obtain
$$
\E \left|\eta_{i-1, j}\right|^{2+\delta} = \E \left|\sum_{\ell=1}^{i-1}\xi_{\ell, j}\right|^{2+\delta}
\leq
C \max \{i, i^{1+ \frac \delta 2}\}
\leq
C n^{1 + \frac \delta 2}.
$$
Therefore,
$$
\sum_{j=1}^d \E |\xi_{i,j} \eta_{i-1, j}|^{2+\delta}
\leq
C\cdot  d \cdot n^{1 + \frac \delta 2}.
$$
To estimate the second term in~\eqref{eq:referee2}, we observe that $\E (\xi_{i,j}\eta_{i-1,j})^2 = \E (\eta_{i-1,j})^2 =  i-1  < n$. It follows that
$$
\left(\sum_{j=1}^d \E (\xi_{i,j}\eta_{i-1,j})^2\right)^{1 + \frac \delta 2}
\leq
(dn)^{1 + \frac \delta 2}.
$$
Altogether we arrive at
$$
\E |Y_i^{(d)}|^{2 + \delta}
\leq C d^{-2-\delta} (dn)^{1 + \frac \delta 2}
=
C (n/d)^{1 + \frac \delta 2},
$$
which proves the claim~\eqref{eq:referee1}.
\end{proof}

\begin{proof}[Proof of Theorem~\ref{theo:distr_norm_model_1}]
By Proposition~\ref{prop:model_1_off_diag_CLT}, the off-diagonal sum satisfies
$$
\frac{Q_n^{(d)}}{\sqrt{2n^2/d}} \todistrd \mathrm{N}(0,1).
$$
To derive a distributional limit theorem for the diagonal sum, we observe that
\begin{equation}\label{eq:diagonal_sum_model_1}
T_n^{(d)} - n =  \sum_{i=1}^{n} (\|X_i^{(d)}\|^2 - 1) = \frac 1d \sum_{i=1}^n \sum_{j=1}^d (\xi_{i,j}^2 - 1)
\end{equation}
Recall the assumption $\E \xi^4 <\infty$. Applying the classical CLT  to the right-hand side of~\eqref{eq:diagonal_sum_model_1} yields
$$
\frac{T_n^{(d)} - n}{\sqrt{n/d}} \todistrd \mathrm{N}(0, \Var [\xi^2]).
$$
Since $\sqrt {n/d} =o(\sqrt{2n^2/d})$, the fluctuations of the off-diagonal sum $Q_n^{(d)}$ dominate. 
\end{proof}

\begin{proof}[Proof of Theorem~\ref{theo:distr_norm_model_1_stable}]
By~\eqref{eq:diagonal_sum_model_1} and~\eqref{eq:alpha_stable_domain_attraction}, we have
$$
\frac{T_n^{(d)} - n}{d^{-1}(nd)^{1/\alpha} L(nd)} = \frac 1{(nd)^{1/\alpha} L(nd)} \sum_{i=1}^n \sum_{j=1}^d (\xi_{i,j}^2 - 1)
\todistrd
\zeta_\alpha.
$$
The normalizing sequence for $T_n^{(d)}-n$ is thus $\tau_n^{(d)} =d^{-1}(nd)^{1/\alpha} L(nd)$.  If, for some $\delta>0$ and all sufficiently large $d$,  $n > d^{\frac{2 - \alpha}{2\alpha  - 2} + \delta}$, respectively, $n < d^{\frac{2 - \alpha}{2\alpha  - 2} - \delta}$, then $\sqrt {n^2/d} = o(\tau_n^{(d)})$ (meaning that the fluctuations of $T_n^{(d)}$ dominate), respectively, $\tau_n^{(d)} =o(\sqrt{2n^2/d})$ (meaning that the fluctuations of $Q_n^{(d)}$ dominate).

It is also clear that, in fact,  a more precise result has  been deduced. Namely, if
$$
\lim_{d\to\infty}n^{1/\alpha-1}d^{1/\alpha-1/2}L(nd)=0,
$$
then the convergence in Part (a) holds true, whereas if the above limit is equal to $+\infty$, the convergence in Part (b) holds true.
\end{proof}


\subsection{Model 2: Proofs of Theorems~\ref{theo:distr_norm_model_2} and~\ref{theo:distr_norm_model_2_stable}}
The main difficulty is again to prove the CLT for the off-diagonal sum.
\begin{prop}\label{prop:model_2_off_diag_CLT}
In addition to the setting of Section~\ref{subsec:model_2} suppose that $\E R^{2+ \delta} <\infty$ for some $\delta>0$. Then,
$$
\frac{Q_n^{(d)}}{\sqrt{2n^2/d}} \todistrd \mathrm{N}(0,1).
$$
\end{prop}
\begin{proof}
We again apply the martingale central limit theorem.
Condition~\eqref{eq:martingale_CLT_cond_1} of Theorem~\ref{theo:martingale_CLT} has been verified in Lemma~\ref{lem:cond_1_mart_CLT_verification}. We shall verify the Lyapunov condition~\eqref{eq:martingale_CLT_cond_lyapunov_1} which takes the form
\begin{equation}\label{eq:model_2_proof_lyapunov}
\frac 1 {(n/\sqrt d)^{2+\delta}}  \sum_{i=1}^n \E \left[|Y_i^{(d)}|^{2+\delta} | \mathcal F_{i-1}^{(d)}\right] \topr 0.
\end{equation}
Recall from Section~\ref{subsec:model_2} that $X_i = R_i U_i^{(d)}$, where $R_i\geq 0$, $U_i^{(d)}$ and $S_{i-1}^{(d)}$ are independent. It follows that
\begin{align*}
\E [|Y_i^{(d)}|^{2 + \delta}| \mathcal F_{i-1}^{(d)}]
&=
\E [|\langle X_i, S_{i-1}^{(d)}\rangle|^{2+\delta}| \mathcal F_{i-1}^{(d)}]
=
\E [|\langle R_i U_i^{(d)}, S_{i-1}^{(d)}\rangle|^{2+\delta}| \mathcal F_{i-1}^{(d)}]\\
&=
\E |R|^{2+\delta} \cdot  \E [|\langle U_i^{(d)}, S_{i-1}^{(d)}\rangle|^{2+\delta}| \mathcal F_{i-1}^{(d)}]
=
\E |R|^{2+\delta} \cdot  |S_{i-1}^{(d)}|^{2+\delta} \cdot  \E |\langle U_i^{(d)}, e_1\rangle|^{2+\delta}
\leq
C \cdot  d^{-1 - \frac \delta 2} \cdot |S_{i-1}^{(d)}|^{2+\delta}
,
\end{align*}
where in the penultimate step we used the isotropy of $U_i^{(d)}$. In the last step, we used that $\sqrt {d} \langle U_i^{(d)}, e_1\rangle$ converges to the standard normal distribution together with all moments~\cite[Theorem~1]{stam} and,  consequently, $\E |\langle U_i^{(d)}, e_1\rangle|^{2+\delta} \leq C d^{-1 - \frac \delta 2}$.
For the Lyapunov sum we obtain the estimate
\begin{align*}
\sum_{i=1}^n \E [|Y_i^{(d)}|^{2+\delta} | \mathcal F_{i-1}^{(d)}]
\leq
C  \cdot  d^{-1 - \frac \delta 2} \cdot \sum_{i=1}^n |S_{i}^{(d)}|^{2+\delta}.
\end{align*}
We know from Theorem~\ref{theo:FWLLN} that the event
$$
A_n^{(d)}:= \left\{ \max_{i\in\{1,\ldots,n\}} \left|\|S_i^{(d)}\|^2 - i\right| \leq  n \right\}
$$
satisfies $\lim_{d\to\infty} \mathbb P[A_n^{(d)}] = 1$. So, on the event $A_n^{(d)}$ we have $\|S_i^{(d)}\|^2 \leq 2n$ for all $i=1,\dots, n$ and $\sum_{i=1}^{n-1} |S_{i}^{(d)}|^{2+\delta} \leq n^{2 + (\delta/2)}$. It follows that, on $A_n^{(d)}$,
\begin{equation}\label{eq:prop54_new_num}
\sum_{i=1}^n \E [|Y_i^{(d)}|^{2+\delta} | \mathcal F_{i-1}^{(d)}]
\leq
C  \cdot  d^{-1 - \frac \delta 2} \cdot \sum_{i=1}^n |S_{i}^{(d)}|^{2+\delta}
\leq
C \cdot d^{-1 - \frac \delta 2} \cdot n^{2 + \frac \delta2}
=
o((n/\sqrt d)^{2+\delta}).
\end{equation}
This proves~\eqref{eq:model_2_proof_lyapunov}.
\end{proof}

\begin{proof}[Proof of Theorem~\ref{theo:distr_norm_model_2}]
By the classical CLT and Proposition~\ref{prop:model_2_off_diag_CLT},
\begin{equation}\label{eq:diagonal_sum_model_2}
\frac{T_n^{(d)} - n}{\sqrt n} =  \frac {\sum_{i=1}^{n} (\|X_i^{(d)}\|^2 - 1)}{\sqrt n} = \frac {\sum_{i=1}^n  (R_i^2 - 1)} {\sqrt n}  \todistrd \mathrm{N}(0, \Var [R^2]),
\quad
\frac{Q_n^{(d)}}{\sqrt{2n^2/d}} \todistrd \mathrm{N}(0,1).
\end{equation}

\vspace*{2mm}
\noindent
\textit{Proof of (a):} If $n/d\to 0$, then $\sqrt{2n^2/d} = o(\sqrt {n})$ and (since $\Var [R^2]\neq 0$) the diagonal sum $T_n^{(d)}$ dominates.

\vspace*{2mm}
\noindent
\textit{Proof of (b):} If $n/d \to \infty$, then $\sqrt {n} = o(\sqrt{2n^2/d})$ and the off-diagonal sum $Q_n^{(d)}$ dominates. The same conclusion applies if $R = 1$ is deterministic since then the diagonal sum equals $n$.

\vspace*{2mm}
\noindent
\textit{Proof of (c).} The proof in the ``critical case'' when $n\sim \gamma d$ follows essentially the same idea as described in Sections~\ref{subsec:general_strategy} and~\ref{subsec:CLT_off_diagonal_sum},  but requires more refined estimates. We start with the decomposition
$$
\frac{\|S_n^{(d)}\|^2 - n}{\sqrt n} = \sum_{i=1}^n \Delta_i^{(d)},
\qquad
\Delta_i^{(d)} := \frac {Y_i^{(d)}} {\sqrt n},
\qquad
Y_i^{(d)} := R_i^2 - 1 + 2 R_i \langle U_i^{(d)}, S_{i-1}^{(d)}\rangle.
$$
The sequence $\Delta_1^{(d)}, \dots, \Delta_n^{(d)}$ forms a martingale difference since
$$
\E [Y_i^{(d)} | \mathcal F_{i-1}] = \E [ R_i^2 - 1 | \mathcal F_{i-1}]  +   \E R_i \cdot \E [\langle U_i^{(d)}, S_{i-1}^{(d)}\rangle| \mathcal F_{i-1}] = 0,
$$
where we used that $\E \langle U_i^{(d)}, x\rangle =0$ for every fixed vector $x\in \R^d$. The latter relation and $\E [\langle U_i^{(d)}, x\rangle^2] = \|x\|^2/d$ imply
$$
\E [(Y_i^{(d)})^2 | \mathcal F_{i-1}]
=
\E(R_i^2-1)^2 + 4 \E [R_i (R_i^2-1)] \E [\langle U_i^{(d)}, S_{i-1}^{(d)}\rangle |\mathcal F_{i-1}] + 4 \E [\langle U_i^{(d)}, S_{i-1}^{(d)}\rangle^2 |\mathcal F_{i-1}]\\
=
\Var (R^2)  + \frac 4d \|S_{i-1}^{(d)}\|^2.
$$
Thus, it follows that
$$
\sum_{i=1}^n \E \left[\left(\Delta_i^{(d)}\right)^2| \mathcal F_{i-1}^{(d)}\right]
=
\Var (R^2) + \frac 4{d n} \sum_{i=1}^n \|S_{i-1}^{(d)}\|^2 \topr \Var (R^2) + 2 \gamma,
$$
where we utilized that $\frac 2 {n^2}\sum_{i=1}^n \|S_{i-1}^{(d)}\|^2$ converges in probability to $1$, which can be verified in the same way as in the proof of Lemma~\ref{lem:cond_1_mart_CLT_verification}. It remains to check the Lindeberg condition~\eqref{eq:martingale_CLT_cond_2} which takes the following form. For every $\eps>0$,
\begin{equation}\label{eq:lindeberg_critical_model2}
\frac 1n \sum_{i=1}^n \E \left[\left(R_i^2 - 1 + 2 R_i \langle U_i^{(d)}, S_{i-1}^{(d)}\rangle\right)^2 \ind_{\left\{|R_i^2 - 1 + 2 R_i \langle U_i^{(d)}, S_{i-1}^{(d)} \rangle| \geq \eps \sqrt n\right\}}\Big | \mathcal F_{i-1}^{(d)}\right] \topr  0.
\end{equation}
By the estimate $(a+b)^2 \leq 2 a^2 + 2b^2$  and Markov's inequality it suffices to verify the following claims:
\begin{align}
&\lim_{d\to\infty} \frac 1n \sum_{i=1}^n \E \left[(R_i^2 - 1)^2 \ind_{\left\{|R_i^2 - 1| \geq \frac 12 \eps \sqrt n \right\}}\right] = 0,
\label{eq:lindeberg_critical_1_model_2}\\
&\lim_{d\to\infty} \frac 1n \sum_{i=1}^n \E \left[(R_i^2 - 1)^2  \ind_{\left\{|R_i \langle U_i^{(d)}, S_{i-1}^{(d)} \rangle| \geq \frac 14
 \eps \sqrt n\right\}}\right] = 0,
 \label{eq:lindeberg_critical_2_model_2}\\
&\lim_{d\to\infty} \frac 1n \sum_{i=1}^n \E \left[R_i^2 \langle U_i^{(d)}, S_{i-1}^{(d)}\rangle^2 \ind_{\left\{|R_i^2 - 1| \geq \frac 12\eps \sqrt n\right\}}\right] = 0,
\label{eq:lindeberg_critical_3_model_2}\\
&\frac 1n \sum_{i=1}^n \E \left[R_i^2 \langle U_i^{(d)}, S_{i-1}^{(d)}\rangle^2 \ind_{\left\{|R_i \langle U_i^{(d)}, S_{i-1}^{(d)} \rangle| \geq \frac 14
 \eps \sqrt n\right\}}\Big | \mathcal F_{i-1}^{(d)}\right] \topr  0.
 \label{eq:lindeberg_critical_4_model_2}
\end{align}
Condition~\eqref{eq:lindeberg_critical_1_model_2} is fulfilled by the monotone convergence theorem since the $R_i$'s are independent copies of $R$ and $\E R^4 <\infty$. To prove the remaining conditions we first observe that
\begin{equation}\label{eq:proof_lindeberg_model_2_scalar_U_S_expect}
\E \left[\langle U_i^{(d)}, S_{i-1}^{(d)}\rangle^2\right]
=
\E \left[ \E \left[\langle U_i^{(d)}, S_{i-1}^{(d)}\rangle^2 | S_{i-1}^{(d)}\right]\right]
=
\frac 1d \E \|S_{i-1}^{(d)}\|^2
=
\frac {i-1}{d}
\leq
\frac nd
\leq C,
\end{equation}
for all $i\in \{1,\dots, n\}$. Recall also that $R_i$, $U_i^{(d)}$ and $S_{i-1}^{(d)}$ are independent. To prove~\eqref{eq:lindeberg_critical_2_model_2}, note that
\begin{align*}
\E \left[(R_i^2 - 1)^2  \ind_{\left\{|R_i \langle U_i^{(d)}, S_{i-1}^{(d)} \rangle| \geq \frac 14
\eps \sqrt n\right\}}\right]
&\leq
\E \left[(R_i^2 - 1)^2  \ind_{\left\{R_i \geq \frac{1}{4}
\eps  n^{1/4} \right\}}\right]
+
\E \left[(R_i^2 - 1)^2  \ind_{\left\{|\langle U_i^{(d)}, S_{i-1}^{(d)} \rangle| \geq n^{1/4}\right\}}\right]\\
&\leq
\E \left[(R^2 - 1)^2  \ind_{\left\{|R| \geq \frac{1}{4}
\eps  n^{1/4} \right\}}\right]
+
\E \left[(R^2 - 1)^2\right]  \mathbb P\left[|\langle U_i^{(d)}, S_{i-1}^{(d)} \rangle| \geq n^{1/4}\right].
\end{align*}
Both summands on the right-hand side go to $0$ (for the second summand this follows from~\eqref{eq:proof_lindeberg_model_2_scalar_U_S_expect} and Markov's inequality).
To prove condition~\eqref{eq:lindeberg_critical_3_model_2}, we note that
$$
\E \left[R_i^2 \langle U_i^{(d)}, S_{i-1}^{(d)}\rangle^2 \ind_{\left\{|R_i^2 - 1| \geq \frac 12\eps \sqrt n\right\}}\right]
=
\E \left[R_i^2  \ind_{\left\{|R_i^2 - 1| \geq \frac 12\eps \sqrt n\right\}}\right]
\cdot
\E \left[\langle U_i^{(d)}, S_{i-1}^{(d)}\rangle^2\right]
\leq
\E \left[R^2  \ind_{\left\{|R^2 - 1| \geq \frac 12\eps \sqrt n\right\}}\right]
\cdot
C
$$
and observe that the right-hand side goes to $0$ as $n\to\infty$ by the monotone convergence theorem.
To prove~\eqref{eq:lindeberg_critical_4_model_2}, we argue as follows:
\begin{align*}
&\E \left[R_i^2 \langle U_i^{(d)}, S_{i-1}^{(d)}\rangle^2 \ind_{\left\{|R_i \langle U_i^{(d)}, S_{i-1}^{(d)} \rangle| \geq \frac 14
 \eps \sqrt n\right\}}\Big | \mathcal F_{i-1}^{(d)}\right]\\
&\qquad \leq
\E \left[R_i^2 \langle U_i^{(d)}, S_{i-1}^{(d)}\rangle^2 \ind_{\left\{|R_i| \geq \frac 14
 \eps n^{1/4}\right\}}\Big | \mathcal F_{i-1}^{(d)}\right]
+
\E \left[R_i^2 \langle U_i^{(d)}, S_{i-1}^{(d)}\rangle^2 \ind_{\left\{|\langle U_i^{(d)}, S_{i-1}^{(d)} \rangle| \geq n^{1/4}\right\}}\Big | \mathcal F_{i-1}^{(d)}\right]\\
&\qquad  \leq
\frac 1d \|S_{i-1}^{(d)}\|^2\cdot  \E \left[R^2\ind_{\left\{|R| \geq \frac 14
 \eps n^{1/4}\right\}}\right]
+
\E \left[\langle U_i^{(d)}, S_{i-1}^{(d)}\rangle^2 \ind_{\left\{|\langle U_i^{(d)}, S_{i-1}^{(d)} \rangle| \geq n^{1/4}\right\}}\Big | \mathcal F_{i-1}^{(d)}\right].
\end{align*}
The expectation of the first summand can be bounded above by $C \E [R^2\ind_{\{|R| \geq \frac 14
 \eps n^{1/4}\}}]$ uniformly over $i\in \{1,\dots,n\}$, which goes to $0$  by the monotone convergence theorem.
To bound the second summand, we observe that, for every $\delta>0$,
$$
\frac 1n \sum_{i=1}^n \E \left[\langle U_i^{(d)}, S_{i-1}^{(d)}\rangle^2 \ind_{\left\{|\langle U_i^{(d)}, S_{i-1}^{(d)} \rangle| \geq n^{1/4}\right\}}\Big | \mathcal F_{i-1}^{(d)}\right]
\leq
n^{-1 - \frac \delta 4} \sum_{i=1}^n \E \left[|\langle U_i^{(d)}, S_{i-1}^{(d)}\rangle|^{2+\delta} \Big | \mathcal F_{i-1}^{(d)}\right]
\leq
C\cdot n^{-1 - \frac \delta 4} \cdot d^{-1 - \frac \delta 2} \cdot n^{2 + \frac \delta2},
$$
where the last inequality holds on the event $A_n^{(d)}$, as we have shown in the proof of Proposition~\ref{prop:model_2_off_diag_CLT}; see formula~\eqref{eq:prop54_new_num}. The right-hand side goes to $0$ if $n\sim \gamma d$, and we have $\mathbb P[A_n^{(d)}] \to 1$, which completes the verification of~\eqref{eq:lindeberg_critical_4_model_2} and the Lindeberg condition~\eqref{eq:lindeberg_critical_model2}. An appeal to the martingale CLT stated in Theorem~\ref{theo:martingale_CLT} completes the proof of Part~(c).
\end{proof}

\begin{proof}[Proof of Theorem~\ref{theo:distr_norm_model_2_stable}]
It follows from~\eqref{eq:alpha_stable_domain_attraction_model_2} and Proposition~\ref{prop:model_2_off_diag_CLT} that
\begin{equation}\label{eq:proof_model_2_stable}
\frac{T_n^{(d)} - n}{n^{1/\alpha} L(n)}
=
\frac{R_1^2 + \dots + R_{n}^2 - n}{n^{1/\alpha} L(n)}
\overset{w}{\underset{n\to\infty}\longrightarrow}
\zeta_\alpha\quad\text{and}\quad
\frac{Q_n^{(d)}}{\sqrt{2n^2/d}} \todistrd \mathrm{N}(0,1).
\end{equation}
If, for some $\delta>0$ and all sufficiently large $d$, we have $n > d^{\frac{\alpha}{2\alpha  - 2} + \delta}$, respectively, $n < d^{\frac{\alpha}{2\alpha  - 2} - \delta}$, then $n^{1/\alpha} L(n) = o(\sqrt{2n^2/d})$, respectively,  $\sqrt{2n^2/d} = o(n^{1/\alpha} L(n))$, and the claims of (a) and (b) follow.
A similar observation as at the end of the proof of Theorem~\ref{theo:distr_norm_model_1_stable} applies here. Namely, if
$$
\lim_{d\to\infty}d^{1/2}n^{1/\alpha-1}L(n)=0,
$$
then the convergence in Part (a) holds true, whereas if the above limit is equal to $+\infty$, then the convergence in Part (b) holds true.
\end{proof}

\begin{rem}\label{rem:critical_case_stable_model_2}
In the missing critical case of Theorem~\ref{theo:distr_norm_model_2_stable}, i.e.\ when~\eqref{eq:alpha_stable_domain_attraction_model_2} holds and $\sqrt{2n^2/d} \sim \gamma n^{1/\alpha} L(n)$ for some constant $\gamma \in (0,\infty)$, we conjecture that
\begin{equation}\label{eq:conject_critical_case_stable_model_2}
\frac{\|S_n^{(d)}\|^2 - n}{\sqrt{2n^2/d}} \todistrd N + \gamma^{-1} \zeta_\alpha,
\end{equation}
where $N$ has the standard normal law, $N$ and $\zeta_\alpha$ are independent. Let us explain the intuition behind this conjecture (in fact, similar arguments apply to all cases of Theorems~\ref{theo:distr_norm_model_2} and~\ref{theo:distr_norm_model_2_stable}).
It is known, see Theorem~4 in~\cite{stam}, that for every fixed $m\in \N$, the collection of $m(m-1)/2$ random variables $\sqrt d \cdot (\langle U_i^{(d)}, U_{j}^{(d)}\rangle)_{1\leq i < j \leq m}$ converges in distribution to a collection of i.i.d.\ standard normal variables $(N_{i,j})_{1\leq i < j \leq m}$. This suggests the approximation
\begin{equation}\label{eq:critical_case_model_2_heuristics}
\|S_n^{(d)}\|^2 - n
=
\sum_{i=1}^{n} (R_i^2 - 1)
+
2 \sum_{1\leq i < j \leq n} R_i
R_j \langle U_i^{(d)},U_j^{(d)}\rangle
\approx
\sum_{i=1}^{n} (R_i^2 - 1)
+
\frac 2{\sqrt d} \sum_{1\leq i < j \leq n}
R_i  R_j N_{i,j}.
\end{equation}
Conditionally on $R_1,R_2,\dots$, the distribution of the second term term is centered normal with the variance
$$
\frac 2d \left(\sum_{i=1}^n R_i^2\right)^2 - \frac 2d\sum_{i=1}^n R_i^4.
$$
By the law of large numbers, $(\sum_{i=1}^n R_i^2)^2 \sim n^2$ a.s., while $\sum_{i=1}^n R_i^4 = o(n^2)$ since $R_i^4$ is in the domain of attraction of an $\alpha/2$-stable distribution with $\alpha/2 > 1/2$. Hence, the variance of the normal distribution is asymptotic to $2n^2/d$. We see that the fluctuations of the first term on the right-hand side of~\eqref{eq:critical_case_model_2_heuristics} are determined by the $R_i$'s, while the fluctuations of the second term are determined by the $N_{i,j}$'s only. Hence, these fluctuations are asymptotically independent. Recalling~\eqref{eq:alpha_stable_domain_attraction_model_2} for the first term, we arrive at~\eqref{eq:conject_critical_case_stable_model_2}.
\end{rem}

\subsection{Model 3: Proofs of Theorems~\ref{theo:distr_norm_model_3}, \ref{theo:distr_norm_model_3_stable}, \ref{theo:distr_norm_model_3_simple_RW}}\label{subsec:proofs_model_3}
The random walk described in Model 3 can be coupled with the classical allocation scheme~\cite{kolchin_etal_book} in which $n$ balls are independently dropped into $d$ equiprobable boxes.  Each time a random walk makes a jump along the line spanned by the basis vector $e_j$, we drop a ball into the box with the number $j$.  Let
$$
k_j(\ell) =  \sum_{i=1}^\ell \ind_{\{V_{i}^{(d)} = e_j\}}
$$
be the number of balls in box $j\in \{1,\dots, d\}$ after $\ell\in \N$ balls have been placed into boxes. Let $(R_{i;j})_{i,j\in \N}$ be independent copies of the random variable $R$ and consider independent random walks $(Z_{k;j})_{k\in \N_0}$, $j\in \N$,  defined by
$$
Z_{k; j} := R_{1; j} + \dots + R_{k; j},
\quad
Z_{0; j}:=0,
\quad
k\in \N,
\quad
j\in \N.
$$
Then, it follows from the definition of Model~3 given in Section~\ref{subsec:model_3} that
\begin{equation}\label{eq:rw_as_urn_scheme}
S_{n}^{(d)} = (S_{n,1}^{(d)},\dots, S_{n,d}^{(d)}) \eqdistr \left(Z_{k_1(n);1}, \dots, Z_{k_d(n);d}\right).
\end{equation}
In particular, this shows that  $\|S_n^{(d)}\|^2$ is a particular case of the so-called randomized decomposable statistics whose limit behaviour has been extensively studied. A survey on this topic with pointers to the original literature including~\cite{bykov_ivanov_randomized_decomposable} and the thesis of S.\ I.\ Bykov~\cite{bykov_thesis} can be found in~\cite{mikhailov_survey_allocations}. It would be possible to prove most of Theorems~\ref{theo:distr_norm_model_3} and~\ref{theo:distr_norm_model_3_simple_RW} by verifying the (quite technical) conditions of Theorems 1.2.1 and 1.3.1 in~\cite{mikhailov_survey_allocations} (which are due to S.\ I.\ Bykov), but we prefer to give independent proofs since these are quite simple. We begin with a  CLT for the off-diagonal sum.

\begin{prop}\label{prop:model_3_off_diag_CLT}
In addition to the setting of Section~\ref{subsec:model_3} suppose that $\E R^{2+ \delta} <\infty$ and $n/\sqrt d \to \infty$.  Then,
$$
\frac{Q_n^{(d)}}{\sqrt{2n^2/d}} \todistrd \mathrm{N}(0,1).
$$
\end{prop}
\begin{proof}
We again apply the martingale central limit theorem; see Theorem~\ref{theo:martingale_CLT}.  Its condition~\eqref{eq:martingale_CLT_cond_1}   has been verified in Lemma~\ref{lem:cond_1_mart_CLT_verification}. It suffices to verify the Lyapunov condition~\eqref{eq:martingale_CLT_cond_lyapunov_2} which takes the form
\begin{equation}\label{eq:proof_model_3_lyapunov_sum}
\sum_{i=1}^n \E |Y_i^{(d)}|^{2+\delta} = o(n^{2+\delta}/ d^{1 + \frac \delta 2}),
\end{equation}
where $\delta>0$ is such that $\E R^{2+ \delta} <\infty$. Without loss of generality we assume that $\delta<2$.

Let $i\in \{1,\dots, n\}$ be fixed. By definition of our model, see Section~\ref{subsec:model_3},
\begin{align*}
\E |Y_i^{(d)}|^{2+\delta}
&=
\E |\langle R_i V_i^{(d)}, S_{i-1}^{(d)}\rangle|^{2+\delta}
=
\E |R|^{2+\delta} \cdot  \E |\langle  V_i^{(d)}, S_{i-1}^{(d)}\rangle|^{2+\delta}\\
&=
\E |R|^{2+\delta} \cdot  \frac 1 d \cdot \sum_{j=1}^d  \E |S_{i-1,j}^{(d)}|^{2+\delta}
=
\E |R|^{2+\delta} \cdot    \E |S_{i-1,1}^{(d)}|^{2+\delta}
=
C  \cdot \E |S_{i-1,1}^{(d)}|^{2+\delta},
\end{align*}
where we recall the notation $S_{i-1}^{(d)}= (S_{i-1,1}^{(d)},\dots, S_{i-1,d}^{(d)})$ for the components of $S_{i-1}^{(d)}$. In view of~\eqref{eq:rw_as_urn_scheme}
$$
\E |S_{i-1,1}^{(d)}|^{2+\delta}=\E |Z_{k_1(i-1);1}|^{2+\delta}\leq C \E (k_1(i-1))^{1+\delta/2},
$$
where we used Rosenthal's inequality, see Appendix~\ref{appC}, in the last estimate. Taking everything together, we arrive at
\begin{equation}\label{eq:proof_model_3_lyapunov_est_by_moments_of_bin}
\sum_{i=1}^n \E |Y_i^{(d)}|^{2+\delta}
\leq
C\cdot \sum_{i=0}^{n-1} \E |S_{i,1}^{(d)}|^{2+\delta}
\leq
C \cdot \sum_{i=1}^{n-1} \E (k_1(i))^{1+\frac \delta 2}.
\end{equation}
Note that $k_1(i)$ has a binomial distribution $\mathrm{Bin} (i, 1/d)$. We claim that, for all $i,d\in\N$ and $\delta\in (0,2]$,
\begin{equation}\label{eq:moments_binomial_distr}
\E  (k_1(i))^{1+\frac \delta 2} \leq
C \cdot (i/d) + C \cdot (i/d)^{1+\frac \delta 2}.
\end{equation}
Observe that $\E k_1(i) = i/d$. Using the inequality $|a+b|^{1+\frac \delta 2} \leq 2^{\delta/2} (|a|^{1+\frac \delta 2} + |b|^{1+\frac \delta 2})$, we obtain
$$
\E  (k_1(i))^{1+\frac \delta 2}
=
\E |k_1(i) - i/d   + i/d|^{1+\frac \delta 2}
\leq
2^{\delta/2}
\E |k_1(i) - i/d |^{1+\frac \delta 2} + 2^{\delta/2} (i/d)^{1+\frac \delta 2}.
$$
We can write $k_1(i) - i/d  = \eps_1 + \dots + \eps_i$, where $\eps_1,\dots, \eps_d$ are zero-mean i.i.d.\ with $\mathbb P[\eps_i = 1 - 1/d] = 1/d$, $\mathbb P[\eps_i = -1/d] = 1-1/d$.
By Corollary 8.2 on p.~151 of~\cite{Gut:2005}, we have
$$
\E |k_1(i) - i/d |^{1+\frac \delta 2}=\E |\eps_1 + \dots + \eps_i|^{1+\frac \delta 2} \leq C i \E |\eps_1|^{1+\frac \delta 2} \leq C i/d.
$$
This proves~\eqref{eq:moments_binomial_distr}. Now we can complete the proof of~\eqref{eq:proof_model_3_lyapunov_sum} as follows.
By~\eqref{eq:proof_model_3_lyapunov_est_by_moments_of_bin} and~\eqref{eq:moments_binomial_distr},
$$
\sum_{i=1}^n \E |Y_i^{(d)}|^{2+\delta}
\leq
C \cdot \sum_{i=1}^{n-1} \E (k_1(i))^{1+\frac \delta 2}
\leq
C \cdot n \cdot \left( (n/d)  +  (n/d)^{1+ \frac \delta 2} \right)
=
C\cdot\left(\frac{n^2}d + \frac{n^{2+\frac \delta 2}}{d^{1 + \frac \delta 2}}\right)
=
o\left(\frac{n^{2+\delta}}{d^{1 + \frac \delta 2}}\right)
$$
by the assumption $n/\sqrt d \to \infty$. The proof of~\eqref{eq:proof_model_3_lyapunov_sum} is complete.
\end{proof}

\begin{proof}[Proof of Theorem~\ref{theo:distr_norm_model_3}]
By the classical CLT,
\begin{equation}\label{eq:diagonal_sum_model_3}
\frac{T_n^{(d)} - n}{\sqrt n} =  \frac {1}{\sqrt n}\sum_{i=1}^{n} (\|X_i^{(d)}\|^2 - 1) = \frac 1 {\sqrt n} \sum_{i=1}^n  (R_i^2 - 1) \todistrd \mathrm{N}(0, \Var [R^2]).
\end{equation}

\vspace*{2mm}
\noindent
\textit{Proof of (a).} If $n/d\to 0$, then Lemma~\ref{lem:V_n_d_and_Var_Q_n_d } yields $\Var Q_n^{(d)} = 2n(n-1)/d = o(n)$ and hence
$$
\frac{\|S^{(d)}_{n}\|^2 - n}{\sqrt{n}}
=
\frac{T_n^{(d)} - n}{\sqrt n}  +  \frac{Q_n^{(d)}}{\sqrt{2n(n-1)/d}}  \cdot \frac{\sqrt{2n(n-1)/d}}{\sqrt n}
\todistrd
\mathrm{N}(0, \Var [R^2]).
$$
Note that we did not use Proposition~\ref{prop:model_3_off_diag_CLT}. In particular, we do not need assumptions imposed therein.

\vspace*{2mm}
\noindent
\textit{Proof of (b).}
If $n/d\to\infty$, then $n/\sqrt d \to \infty$ and we can apply Proposition~\ref{prop:model_3_off_diag_CLT}, resulting in
$$
\frac{\|S^{(d)}_{n}\|^2 - n}{\sqrt{2n^2/d}}
=
\frac{T_n^{(d)} - n}{\sqrt n} \cdot \frac{\sqrt n} {\sqrt{2n^2/d}} +  \frac{Q_n^{(d)}}{\sqrt{2n^2/d}}
\todistrd
\mathrm{N}(0, 1).
$$

\vspace*{2mm}
\noindent
\textit{Proof of (c).} The starting point is the decomposition
$$
\frac{\|S_n^{(d)}\|^2 - n}{\sqrt n} = \sum_{i=1}^n \Delta_i^{(d)},
\qquad
\Delta_i^{(d)} := \frac {Y_i^{(d)}} {\sqrt n},
\quad
Y_i^{(d)} := R_i^2 - 1 + 2 R_i \langle V_i^{(d)}, S_{i-1}^{(d)}\rangle.
$$
The sequence $\Delta_1^{(d)}, \dots, \Delta_n^{(d)}$ forms a martingale difference since
$$
\E [Y_i^{(d)} | \mathcal F_{i-1}] = \E [ (R_i^2 - 1) | \mathcal F_{i-1}]  + 2\E R_i \cdot \E [\langle V_i^{(d)}, S_{i-1}^{(d)}\rangle| \mathcal F_{i-1}] = 0,
$$
where we used that $\E R_i = 0$.  Next we observe that
\begin{align*}
\E [(Y_i^{(d)})^2 | \mathcal F_{i-1}]
&=
\E[(R_i^2-1)^2] + 4 \E [R_i (R_i^2-1)] \E [\langle V_i^{(d)}, S_{i-1}^{(d)}\rangle |\mathcal F_{i-1}] + 4 \E [\langle V_i^{(d)}, S_{i-1}^{(d)}\rangle^2 |\mathcal F_{i-1}]\\
&=
\Var (R^2)  + 4 \E R^3 \cdot \frac 1d (S_{i-1,1}^{(d)} + \dots +  S_{i-1,d}^{(d)})
+
\frac 4d \|S_{i-1}^{(d)}\|^2,
\end{align*}
where we used that  $\E \langle V_i^{(d)}, x\rangle =\frac 1d (x_1 + \dots + x_d)$ and $\E \langle V_i^{(d)}, x\rangle^2 = \|x\|^2/d$ for each fixed vector $x= (x_1,\dots, x_d)\in \R^d$.

Let us check that
\begin{equation}\label{eq:model3_double_sum_to_zero}
\frac 1 {dn} \sum_{i=1}^n \sum_{j=1}^d S_{i-1,j}^{(d)} \topr 0.
\end{equation}
Indeed,
\begin{align*}
\frac 1 {dn} \sum_{i=1}^n \sum_{j=1}^d S_{i-1,j}^{(d)}&=\frac 1 {dn} \sum_{i=1}^n \sum_{j=1}^d \sum_{k=1}^{i-1}R_k \ind_{\{V_k^{(d)}=e_j\}}=\frac 1 {dn} \sum_{i=1}^n \sum_{k=1}^{i-1}R_k= \frac 1 {dn} \sum_{k=1}^{n-1}(n-k) R_k ,
\end{align*}
and the right-hand side converges to zero in probability by Chebyshev's inequality, since the variance of the right-hand side is $O(n^3/(nd)^2) = O(1/d)$.

Formula~\eqref{eq:model3_double_sum_to_zero} together with the fact that
$\frac 2 {n^2}\sum_{i=1}^n \|S_{i-1}^{(d)}\|^2$ converges in probability to $1$, which can be verified in the same way as in the proof of Lemma~\ref{lem:cond_1_mart_CLT_verification}, yield
$$
\sum_{i=1}^n \E \left[\left(\Delta_i^{(d)}\right)^2| \mathcal F_{i-1}^{(d)}\right]
=
\Var (R^2) + 4 \E R^3\cdot \frac 1 {dn} \sum_{i=1}^n \sum_{j=1}^d S_{i-1,j}^{(d)} + \frac 4{d n} \sum_{i=1}^n \|S_{i-1}^{(d)}\|^2 \topr \Var (R^2) + 2 \gamma.
$$

It remains to verify the following Lindeberg condition that implies~\eqref{eq:martingale_CLT_cond_2}. For every $\eps>0$,
\begin{equation}\label{eq:lindeberg_critical_model3}
\lim_{d\to\infty} \frac 1n \sum_{i=1}^n \E \left[\left(R_i^2 - 1 + 2 R_i \langle V_i^{(d)}, S_{i-1}^{(d)}\rangle\right)^2 \ind_{\left\{|R_i^2 - 1 + 2 R_i \langle V_i^{(d)}, S_{i-1}^{(d)} \rangle| \geq \eps \sqrt n\right\}}
\right] =  0.
\end{equation}
The proof proceeds by the same method as in the proof of Theorem~\ref{theo:distr_norm_model_3}~(c) with the following modifications. The analogues of~\eqref{eq:lindeberg_critical_1_model_2}, \eqref{eq:lindeberg_critical_2_model_2}, \eqref{eq:lindeberg_critical_3_model_2}  (with $U_i^{(d)}$ replaced by $V_i^{(d)}$) can be established in the same way as above upon replacing~\eqref{eq:proof_lindeberg_model_2_scalar_U_S_expect} by
\begin{equation}\label{eq:proof_lindeberg_model_3_scalar_V_S_expect}
\E \left[\langle V_i^{(d)}, S_{i-1}^{(d)}\rangle^2\right]
=
\E \left[\E \left[\langle V_i^{(d)}, S_{i-1}^{(d)}\rangle^2 | S_{i-1}^{(d)}\right]\right]
=
\frac 1d \E \|S_{i-1}^{(d)}\|^2
=
\frac {i-1}{d}
\leq
\frac nd
\leq C.
\end{equation}
Instead of~\eqref{eq:lindeberg_critical_4_model_2} we verify the following condition:
\begin{equation}\label{eq:lindeberg_critical_4_model_3}
\lim_{d\to\infty} \frac 1n \sum_{i=1}^n \E \left[R_i^2 \langle V_i^{(d)}, S_{i-1}^{(d)}\rangle^2 \ind_{\left\{|R_i \langle V_i^{(d)}, S_{i-1}^{(d)} \rangle| \geq \frac 14
 \eps \sqrt n\right\}}\right] =  0.
\end{equation}
Take some $\delta >0$. Then,
\begin{align*}
\frac 1n \sum_{i=1}^n \E \left[R_i^2 \langle V_i^{(d)}, S_{i-1}^{(d)}\rangle^2 \ind_{\left\{|R_i \langle V_i^{(d)}, S_{i-1}^{(d)} \rangle| \geq \frac 14
\eps \sqrt n\right\}}\right]
&\leq
\frac 1n  \left(\frac{1}{4} \eps \sqrt n\right)^{-\delta} \sum_{i=1}^n \E \left[|R_i|^{2+\delta} |\langle V_i^{(d)}, S_{i-1}^{(d)}\rangle|^{2+\delta}\right]\\
&\leq
C\cdot  n^{-1-\frac \delta 2}  \sum_{i=1}^n \E \left[|\langle V_i^{(d)}, S_{i-1}^{(d)}\rangle|^{2+\delta}\right]\leq
C\cdot n^{-1-\frac \delta 2} \cdot \left(\frac{n^2}d + \frac{n^{2+\frac \delta 2}}{d^{1 + \frac \delta 2}}\right),
\end{align*}
where the last inequality was established in the proof of Proposition~\ref{prop:model_3_off_diag_CLT}. The right-hand side converges to $0$ in the regime when $n\sim \gamma d$, which proves~\eqref{eq:lindeberg_critical_4_model_3} and completes the verification of the Lindeberg condition~\eqref{eq:lindeberg_critical_model3}. Thus, Part~(c) follow by another appeal to Theorem~\ref{theo:martingale_CLT}.
\end{proof}

\begin{proof}[Proof of Theorem~\ref{theo:distr_norm_model_3_stable}]
By~\eqref{eq:alpha_stable_domain_attraction_model_2}, we have
\begin{equation}\label{eq:proof_model_3_stable_copy}
\frac{T_n^{(d)} - n}{n^{1/\alpha} L(n)}
\todistrd
\zeta_\alpha.
\end{equation}

\vspace*{2mm}
\noindent
\textit{Proof of (a).}
If  $n > d^{\frac{\alpha}{2\alpha  - 2} + \delta}$ for some $\delta>0$ and all sufficiently large $d$, then we also have $n/d\to \infty$ since $\alpha\in (1,2)$. Hence, Proposition~\ref{prop:model_3_off_diag_CLT} applies and $Q_n^{(d)}$ satisfies a CLT with normalization $\sqrt{2n^2/d}$. We have $n^{1/\alpha} L(n) = o(\sqrt{2n^2/d})$, meaning that the off-diagonal fluctuations dominate, and the claim follows.

\vspace*{2mm}
\noindent
\textit{Proof of (b).}
If $n < d^{\frac{\alpha}{2\alpha  - 2} - \delta}$ for all sufficiently large $d$, then $\Var Q_{n}^{(d)} = 2n(n-1)/d$ by~\eqref{eq:V_n_d_Q_n_d_expect} and $\sqrt{2n(n-1)/d} = o(n^{1/\alpha} L(n))$. It follows that
$$
\frac{\|S^{(d)}_{n}\|^2 - n}{n^{1/\alpha} L(n)}
=
\frac{T_n^{(d)} - n}{n^{1/\alpha} L(n)} +  \frac{Q_n^{(d)}}{\sqrt{2n(n-1)/d}} \frac{\sqrt{2n(n-1)/d}} {n^{1/\alpha} L(n)}
\todistrd
\zeta_\alpha,
$$
which proves the claim.
\end{proof}

\begin{proof}[Proof of Theorem~\ref{theo:distr_norm_model_3_simple_RW}]
\textit{Proof of (a):} If $n=o(\sqrt d)$, then by a well-known result on the birthday problem (see, e.g., Example~3.2.5 in~\cite{durrett_book} or p.~42 in~\cite{kolchin_etal_book}), the probability that no box contains $\geq 2$ balls (equivalently, that the vectors  $V_1^{(d)}, \dots, V_n^{(d)}$ are pairwise different) converges to $1$. On this event, we evidently have $\|S_n^{(d)}\|^2 = n$.

\vspace*{2mm}
\noindent
\textit{Proof of (b):} Using the notation introduced at the beginning of the present Section~\ref{subsec:proofs_model_3}, we can write
\begin{equation}\label{eq:S_n_2_decomposition_for_simple_RW_balls}
\|S_n^{(d)}\|^2 = \sum_{j=1}^d Z^2_{k_j(n); j} = \sum_{j=1}^d Z^2_{k_j(n); j} \ind_{\{k_j(n) \geq 3\}} + \sum_{j=1}^d Z^2_{k_j(n); j} \ind_{\{k_j(n) \leq 2\}}.
\end{equation}
Let us show that the first sum (which is the total contribution of the boxes containing at least $3$ balls) converges in distribution to $0$. The expectation of this term is given by
\begin{align*}
\E \left[\sum_{j=1}^d Z^2_{k_j(n); j} \ind_{\{k_j(n) \geq 3\}}\right]
=
d \E \left[ Z^2_{k_1(n); 1} \ind_{\{k_1(n) \geq 3\}}\right]
=
d \sum_{k=3}^\infty  \E \left[Z^2_{k; 1} \ind_{\{k_1(n) = k\}}\right]
=
d \sum_{k=3}^\infty  k \mathbb P[k_1(n) = k].
\end{align*}
Using that $k_1(n)$ has a binomial distribution $\mathrm{Bin} (n,1/d)$ and that $n\leq 2c\sqrt d$ for sufficiently large $d$, we can write
\begin{equation}\label{eq:proof_simple_RW_estimate_contr_3_balls}
\E \left[\sum_{j=1}^d Z^2_{k_j(n); j} \ind_{\{k_j(n) \geq 3\}}\right]=
d \sum_{k=3}^\infty  k \binom nk \frac 1 {d^k} \left(1 - \frac 1d\right)^{n-k}
\leq
d \sum_{k=3}^\infty  k \frac{n^k}{k!} \frac 1 {d^k}
\leq
d \sum_{k=3}^\infty  \frac{(2c \sqrt d)^k}{(k-1)!} \frac 1 {d^k}
\leq
\frac{1}{\sqrt{d}}\sum_{k=3}^\infty  \frac{(2c)^k}{(k-1)!}
\leq
C d^{-1/2},
\end{equation}
which converges to $0$, as $d\to\infty$.
Let us now analyze the second sum in~\eqref{eq:S_n_2_decomposition_for_simple_RW_balls}. For each $k\in \N_0$ let  $\mu_k(n) := \sum_{j=1}^d \ind_{\{k_j(n) = k\}}$ be the number of boxes containing $k$ balls. It is known, see Example~2 on pp.~14--15 in~\cite{arratia_goldstein_gordon} or Theorems~3, 5 on pp.~67--68 in~\cite{kolchin_etal_book},  that in the regime when $n\sim c \sqrt d$ with $c\in (0,\infty)$,
\begin{equation}\label{eq:mu_2_conv_simple_RW}
\mu_2(n) \todistrd  \mathrm{Poi} (c^2/2).
\end{equation}
The number of boxes containing at least $3$ balls is denoted by
$$
\mu_{\geq 3} (n) := \sum_{k=3}^\infty \mu_k(n) = \sum_{j=1}^d \ind_{\{k_j(n) \geq 3\}}.
$$
Almost the same estimate as in~\eqref{eq:proof_simple_RW_estimate_contr_3_balls} shows that $\E \mu_{\geq 3}(n) \to 0$ and hence $\mu_{\geq 3}(n)\to 0$ in probability. If some box $j$ contains $1$ ball, then the corresponding contribution $Z^2_{1;j}$ is $1$.  If some box $j$ contains $2$ balls, then  $Z^2_{2;j}$ is either $(1+1)^2 = (-1-1)^2 = 4$ or $(+1-1)^2 = (-1+1)^2 = 0$, both possibilities having probability $1/2$. Denoting by $\eta_1,\eta_2,\dots$ i.i.d.\ random variables with $\mathbb P[\eta_\ell = 4]= \mathbb P[\eta_\ell = 0] = 1/2$, we can write
\begin{equation}\label{eq:proof_model_3_boxes_with_1_or_2_balls}
\sum_{j=1}^d Z^2_{k_j(n); j} \ind_{\{k_j(n) \leq 2\}} - n
\eqdistr
\mu_1(n) + \sum_{\ell = 1}^{\mu_2(n)} \eta_\ell  - n
=
\sum_{\ell = 1}^{\mu_2(n)} \eta_\ell  - \mu_2(n) - \mu_{\geq 3}(n)
=
\sum_{\ell = 1}^{\mu_2(n)} (\eta_\ell-1) - \mu_{\geq 3}(n).
\end{equation}
Since the random variables $\eta_\ell-1$ take values $3$ and $-1$ with probability $1/2$ each and since $\mu_2(n)$ converges in distribution to $\mathrm{Poi} (c^2/2)$ by~\eqref{eq:mu_2_conv_simple_RW}, it follows that the right-hand side of~\eqref{eq:proof_model_3_boxes_with_1_or_2_balls} converges in distribution to $3 P' - P''$, where $P',P''$ are independent and both have a Poisson distribution with parameter $c^2/4$.

\vspace*{2mm}
\noindent
\textit{Proof of (c):} For the simple random walk, the diagonal sum is deterministic:  $T_n^{(d)} = n$. If $n/ \sqrt d\to \infty$, then the off-diagonal sum $Q_n^{(d)}$ satisfies a CLT by Proposition~\ref{prop:model_3_off_diag_CLT}, and the claim follows.
\end{proof}

\begin{appendix}

\section{A law of large numbers}\label{appA}
In Section~\ref{sec:proof_LLN_wiener_spiral} we used the following version of the weak law of large numbers for triangular arrays.

\begin{lemma}\label{lem:wlln}
Assume that $(\theta_{d,i})_{i\in\N}$, for every $d\in \N$, is a sequence of independent copies of a random variable $\theta_d$. Suppose that $\E \theta_d=1$ for all $d\in\N$, and the family $(\theta_d)_{d\in\N}$ is uniformly integrable. Then, for every integer sequence $n=n(d)$ such that $n(d)\to\infty$, as $d\to\infty$,
$$
\frac{1}{n}\sum_{i=1}^{n}\theta_{d,i}~\topr~1.
$$
\end{lemma}
\begin{proof}
The proof is standard and goes along the same lines as the proof of Theorem on p.~105 in~\cite{Gnedenko+Kolmogorov:1968}. Put $\theta_{d,i}(n):=\theta_{d,i}\ind_{\{|\theta_{d,i}|\leq n\}}$ and note that
$$
\mmp\left\{\sum_{i=1}^{n}\theta_{d,i}\neq  \sum_{i=1}^{n}\theta_{d,i}(n)\right\}\leq n\mmp\{|\theta_d|\geq n\}\leq \E\left(|\theta_d|\ind_{\{|\theta_d|\geq n\}}\right)\leq  \sup_{d\in\N}\E\left(|\theta_d|\ind_{\{|\theta_d|\geq n\}}\right).
$$
The left-hand side converges to zero, as $n\to\infty$, by the definition of the uniform integrability of the family $(\theta_d)_{d\in\N}$. By the same reasoning,
$$
\lim_{d\to\infty}\frac{1}{n}\sum_{i=1}^{n}\E\theta_{d,i}(n)=1.
$$
Thus, by Chebyshev's inequality, it remains to show that
$$
\Var \left(\frac{1}{n}\sum_{i=1}^{n}\theta_{d,i}(n)\right)~\to~0,\quad d\to\infty.
$$
Clearly, it suffices to check
\begin{equation}\label{eq:wlln_proof1}
\lim_{d\to\infty}\frac{\E \theta^2_{d,1}(n)}{n}=0.
\end{equation}
Observe that by Fubini's theorem
\begin{multline*}
\E \theta^2_{d,1}(n)=2\int_{[0,n]}\left(\int_0^y s {\rm d}s\right) \mmp\{|\theta_{d}|\in {\rm d}y\}=2\int_{0}^{n}s\left(\int_{[s,n]} \mmp\{|\theta_{d}|\in {\rm d}y\}\right){\rm d}s\\
 \leq 2\int_0^n s \mmp\{|\theta_{d}|\geq s\}{\rm d}s\leq 2\int_0^n \E \left(|\theta_{d}|\ind_{\{|\theta_{d}|\geq s\}}\right){\rm d}s \leq 2\int_0^n \sup_{d\in\N }\E \left(|\theta_{d}|\ind_{\{|\theta_{d}|\geq s\}}\right){\rm d}s.
\end{multline*}
Thus,~\eqref{eq:wlln_proof1} follows by an appeal to L'H\^{o}pital's rule.
\end{proof}

\section{Martingale central limit theorem}\label{appB}
In Section~\ref{sec:proofs_CLT} we used a central limit theorem for martingale triangular arrays which can be found in~\cite[Theorem~2]{gaenssler_etal} and, in a slightly less general setting, in~\cite[Theorem~2]{brown}.
To state it, we need some notation.  For every $d\in \N$, let $\mathcal{G}_{0}^{(d)} \subset \mathcal{G}_{1}^{(d)} \subset \dots \subset \mathcal{G}_{n}^{(d)}$ be a filtration on a probability space $(\Omega, \mathcal{G}, \mathbb P)$, where  $n=n(d)$ is a sequence of positive integers such that $n(d)\to \infty$, as  $d\to\infty$. The random variables $\Upsilon_1^{(d)}, \dots, \Upsilon_{n}^{(d)}$ are said to form an \textit{array of martingale differences} if $\Upsilon_{i}^{(d)}$ is $\mathcal{G}_i^{(d)}$-measurable and $\E [\Upsilon_i^{(d)}| \mathcal {G}_{i-1}^{(d)}] = 0$ for all $i=1,\dots, n$. The following result~\cite[Theorem~2]{gaenssler_etal} provides sufficient conditions under which the CLT holds for the random variables
$
\Upsilon_1^{(d)} +  \dots +  \Upsilon_{n}^{(d)}.
$
\begin{theorem}\label{theo:martingale_CLT}
In the setting just described, assume that the following conditions hold:
\begin{enumerate}
\item[(a)] The variables $\Upsilon_1^{(d)}, \dots, \Upsilon_{n}^{(d)}$ have finite second moments and
\begin{equation}\label{eq:martingale_CLT_cond_1}
\sum_{i=1}^n \E \left[\left(\Upsilon_i^{(d)}\right)^2 \Big| \mathcal{G}_{i-1}^{(d)}\right] \topr \sigma^2 \in (0,\infty).
\end{equation}
\item[(b)] For every $\eps>0$,
\begin{equation}\label{eq:martingale_CLT_cond_2}
\sum_{i=1}^n\E \left[\left(\Upsilon_i^{(d)}\right)^2 \ind_{\left\{|\Upsilon_i^{(d)}| \geq \eps\right\}}\Big| \mathcal{G}_{i-1}^{(d)}\right] \topr 0.
\end{equation}
Then, $\Upsilon_1^{(d)} +  \dots +  \Upsilon_{n}^{(d)}$
converges weakly to the normal distribution $\mathrm{N}(0,\sigma^2)$, as $d\to\infty$.
\end{enumerate}
\end{theorem}
The Lindeberg-type condition stated in~\eqref{eq:martingale_CLT_cond_2} follows from the following Lyapunov-type condition: for some $\delta>0$,
\begin{equation}\label{eq:martingale_CLT_cond_lyapunov_1}
\sum_{i=1}^n\E \left[\left|\Upsilon_i^{(d)}\right|^{2+\delta} \Big| \mathcal{G}_{i-1}^{(d)}\right] \topr 0.
\end{equation}
To prove this implication, observe that $(\Upsilon_i^{(d)})^2 \ind_{\{|\Upsilon_i^{(d)}| \geq \eps\}} \leq \eps^{-\delta} |\Upsilon_i^{(d)}|^{2+\delta}$. Note also that, by the Markov inequality, condition~\eqref{eq:martingale_CLT_cond_lyapunov_1} follows from
\begin{equation}\label{eq:martingale_CLT_cond_lyapunov_2}
\lim_{d\to\infty} \sum_{i=1}^n \E \left[\left|\Upsilon_i^{(d)}\right|^{2+\delta} \right] =  0.
\end{equation}

\section{Rosenthal inequality}\label{appC}
In Section~\ref{sec:proofs_CLT} we frequently used the following Rosenthal inequality; see Theorem~9.1 on p.~152 in~\cite{Gut:2005}.
\begin{theorem}\label{theo:rosenthal_ineq}
For every $\delta>0$ there is a universal constant $B = B(\delta)$ such that the following holds: If $Z_1,\dots, Z_d$ are independent random variables with $\E [Z_j]=0$ and $\E [|Z_j|^{2+\delta}] <\infty$ for all $j=1,\dots, d$, then
$$
\E \left|\sum_{j=1}^d Z_j\right|^{2+\delta} \leq B \max \left\{\sum_{j=1}^d  \E |Z_j|^{2+\delta}, \left(\sum_{j=1}^d \E [Z_j^2]
\right)^{(2+\delta)/2} \right\}.
$$
\end{theorem}
If the random variables $Z_1,\dots, Z_d$ are identically distributed, then the Rosenthal inequality yields
\begin{equation}\label{eq:rosenthal_ineq_cor}
\E \left|\sum_{j=1}^d Z_j\right|^{2+\delta} \leq B \max \left\{ d \cdot  \E |Z_1|^{2+\delta}, \left(d \E [Z_1^2]
\right)^{(2+\delta)/2} \right\}
\leq
C d^{1+ \frac \delta 2},
\end{equation}
for all $d\in \N$, with a constant $C$ depending only on the distribution of $Z_1$ and $\delta$. Alternatively, this inequality follows from Corollary 8.2 on page~151 in~\cite{Gut:2005}.

\section{Proof of Lemma~\ref{lem:isometry_extension_to_affine_hull}}\label{appD}

Let $x_0\in K_1$ be any point. Consider two arbitrary elements $z, z'\in \aff K_1$. They can be represented as affine combinations $z=\sum_{i=1}^n \lambda_i x_i$ and $z' = \sum_{i=1}^n \lambda_i' x_i'$, where $x_1,\ldots, x_n\in K_1$, $\lambda_1,\ldots, \lambda_n\in \R$, $\lambda_1',\ldots, \lambda_n'\in \R$ with $\sum_{i=1}^n \lambda_i = \sum_{i=1}^n \lambda_i' = 1$. Let $\mu_i:= \lambda_i - \lambda_i'$.  Then, $\sum_{i=1}^n \mu_i = 0$ and
\begin{multline*}
\|z-z'\|^2
=
\left\|\sum_{i=1}^n \mu_i x_i\right\|^2
=
\left\|\sum_{i=1}^n \mu_i (x_i - x_0) \right\|^2
=
\sum_{i,j=1}^n \mu_i\mu_j \langle x_i-x_0, x_j-x_0\rangle\\
=
\frac 12 \sum_{i,j=1}^n \mu_i \mu_j (\|x_i - x_0\|^2 + \|x_j-x_0\|^2 - \|x_i-x_j\|^2).
\end{multline*}
So, we can express the distance between any points $z\in \aff K_1$ and $z'\in \aff K_1$ in terms of all distances of the form $\|x-y\|^2$, where $x,y\in K_1$. Similar argument applies to points in $\aff K_2$ and shows that the distance between $\sum_{i=1}^n \lambda_i J(x_i)$ and $\sum_{i=1}^n \lambda_i' J(x_i')$ can be expressed in terms of all distances of the form $\|J(x)-J(y)\|^2$. Since $J: K_1 \to K_2$ is an isometry, it follows that
\begin{equation}\label{eq:proof_isometry_extends_affine_hulls}
\left \|\sum_{i=1}^n \lambda_i x_i-\sum_{i=1}^n \lambda_i' x_i\right\|^2 = \left \|\sum_{i=1}^n \lambda_i J(x_i)-\sum_{i=1}^n \lambda_i' J(x_i)\right\|^2.
\end{equation}
This shows that we can extend $J$ to an isometry between $\aff K_1$ and $\aff K_2$ by putting $J(\sum_{i=1}^n \lambda_i x_i) = \sum_{i=1}^n \lambda_i J(x_i)$.  This extension is well-defined since if $\sum_{i=1}^n \lambda_i x_i = \sum_{i=1}^n \lambda_i' x_i$ are two different representations of the same element from $\aff K_1$, then $\sum_{i=1}^n \lambda_i J(x_i) = \sum_{i=1}^n \lambda_i' J(x_i)$ by~\eqref{eq:proof_isometry_extends_affine_hulls}.  Since the Hilbert space is complete, this isometry between $\aff K_1$ and $\aff K_2$ can be extended to the isometry between the closures of these sets.  Uniqueness follows from the fact that any isometry between Hilbert spaces is affine.

\end{appendix}

\begin{acks}[Acknowledgments]
The authors would like to thank two anonymous referees for numerous comments and suggestions that led to significant improvement of the manuscript. In particular, the referees pointed out several important references, including~\cite{grundmann2014limit}, where Theorem~\ref{theo:distr_norm_model_2} has been derived for the first time.
\end{acks}

\begin{funding}
ZK was supported by the German Research Foundation under Germany's Excellence Strategy  EXC 2044 -- 390685587, Mathematics M\"unster: Dynamics - Geometry - Structure and  by the DFG priority program SPP 2265 Random Geometric Systems. AM was supported by the Alexander von Humboldt Foundation.
\end{funding}


\bibliographystyle{imsart-number} 
\bibliography{Wiener_Spiral_BIB}

\end{document}